\documentclass[12pt,reqno]{amsart}
\pdfoutput=1
 
\usepackage[utf8]{inputenc}
\usepackage{amsmath, amsfonts, amsthm, amssymb}
\usepackage[margin=1in]{geometry}
\usepackage{asymptote}
\usepackage{hyperref}
\usepackage{float}
\usepackage{stix}
\usepackage{enumerate}

\title[Upper bounds for hitting times]{Upper bounds for hitting times\\ of random walks on sparse graphs}
\author{Dmitri V. Fomin}
\address{Dmitri V. Fomin, NanoSemi Inc., Waltham, MA 02452, USA}
\email{fomin@hotmail.com}
\date{\today}

\newenvironment{proofL}[1][Proof]{\begin{proof}[#1]}{\end{proof}}
\newenvironment{proofP}[1][Proof]{\begin{proof}[#1]}{\end{proof}}

\newtheorem{mthm}{Theorem}[section]
\newtheorem{mlemma}[mthm]{Lemma}
\newtheorem{mcrl}[mthm]{Corollary}
\newtheorem{mprop}[mthm]{Proposition}

\newcommand{\dist}{\mathrm{d}}
\newcommand{\gte}{\geqslant}
\newcommand{\lte}{\leqslant}
\newcommand{\xxx}{\begin{center}*\quad *\quad *\end{center}}
\newcommand{\vvv}{\vspace{5pt}}

\newcommand{\nn}{\nonumber}
\newcommand{\nncr}{\nonumber\\}

\newcommand{\mch}{\mathcal{H}}
\newcommand{\mcl}{\mathcal{L}}
\newcommand{\mcp}{\mathcal{P}}
\newcommand{\mcf}{\mathcal{F}}
\newcommand{\mcn}{\mathcal{N}}
\newcommand{\mcc}{\mathcal{C}}
\newcommand{\mca}{\mathcal{A}\mathbf{1}}
\newcommand{\mcv}{\mathcal{V}}
\newcommand{\mce}{\mathcal{E}}

\newcommand{\mbp}{\mathbb{P}}
\newcommand{\mbv}{\mathbb{V}}
\newcommand{\mbr}{\mathbb{R}}
\newcommand{\mbs}{\mathbb{S}}
\newcommand{\vphfpp}{\vphantom{\frac{p_{YX}}{p_{YZ}}}}

\newcommand{\dmhs}{\,\,}
\newcommand{\edge}[1]{\left[{#1}\right]}
\newcommand{\dir}[1]{\overrightarrow{#1}}

\newcommand{\msection}[1]{\vspace{15pt}\section{#1}\vspace{5pt}}

\begin{document}


\begin{abstract}

We obtain upper bounds (in most cases, sharp) for the hitting times of random walks on finite undirected graphs.  In particular, we show that the maximum hitting time for a simple random walk on a connected graph with $m$ edges is at most~$m^2$.  Similar bounds are given for the settings involving arbitrary edge-weight and edge-cost functions.

Upper bounds of this type are especially useful for sparse graphs. 

\end{abstract}

\maketitle



\msection{Introduction and main results}

Let $G=(\mcv, \mce)$ be a finite undirected connected simple graph with $m$ edges.  Consider a random walk on~$G$, i.e., a Markov chain whose states are the vertices in~$\mcv$ and whose transitions are restricted to the edges in~$\mce$. The \textit{hitting time} $\mch(x, y)$ associated to a pair of vertices $x,y\in\mcv$ is the expected time that the random walk originating at~$x$ takes to reach~$y$ (for the first time). 

\begin{mthm}
\label{th:simple}
For a simple random walk on $G$, and any $x, y \in \mcv$, we have
$$
\mch(x, y) \lte m^2 - (m-d)^2,
$$
where $d=\dist(x,y)$ is the distance between $x$ and~$y$.
\end{mthm}

Theorem~\ref{th:simple} has the following direct corollary. 

\begin{mthm}
\label{thm:sg1i}
For a simple random walk, the maximum hitting time
$\mch(G)\!=\!\displaystyle\max_{x,y\in\mcv}\mch(x, y)$
satisfies $\mch(G) \lte m^2$, with equality reached only for a path graph with $m$ edges.
\end{mthm}

These results generalize to asymmetric random walks (i.e., random walks on edge-weighted graphs) and to random walks with edge-cost functions. 

An \textit{edge-weight function} on $G$ is a mapping $w:\mcv\times\mcv\rightarrow \mbr$ such that $w(x,y)>0$ if $x$ and~$y$ are adjacent, and $w(x,y)=0$ otherwise.  Such a function defines a random walk on $G$ with transition probabilities
$$
p_{xy} = \frac{w(x,y)}{\sum_{z\in\mcv}w(x,z)}. 
$$

Let $\mcn(x)$ denote the set of vertices adjacent to $x\in \mcv$. The \textit{asymmetry}, or \textit{transitional bias}, of~$w$  (and of the corresponding edge-weighted graph and random walk) is defined by  
$$
\tau_w = \max_{x\in\mcv}\max_{y, z\in\mcn(x)}\frac{w(x,y)}{w(x,z)}. 
$$
Obviously $\tau_w=1$ if and only if $w$ is constant on $\mce$, or equivalently the random walk is simple. 

An \emph{edge-cost function} is a function $f:\mce \rightarrow \mbr$. The associated \emph{cost} of a finite walk in $G$ is defined as the sum of costs of all its edges (with multiplicities). If $f(e)=1$ for every edge~$e$, then the cost of a walk is equal to its length. 

The hitting time $\mch^f(x,y) = \mch^f_w(x, y)$ relative to the edge-cost function~$f$ is defined as the expected value of the cost of a random walk starting at~$x$ and stopping at~$y$.  

Our main result is the following generalization of Theorem~\ref{thm:sg1i}. 

\begin{mthm}
Let $f$ be a nonnegative edge-cost function. Then for any $x, y\in\mcv$, we have
$$
\mch^f_w(x, y) \lte \sum_{e\in \mce} \Bigl(1+2\sum_{i=1}^{\dist_y(e)}\tau_w^i \Bigr)\cdot f(e), 
$$
where $\dist_y(e)$ is the minimum distance between $y$ and the endpoints of~$e$.
\end{mthm}



\msection{Definitions, notation and some basic facts}

This section will repeat some definitions from the previous section, with necessary expansion and clarification. 

We will consider a finite undirected connected simple graph $G = (\mcv, \mce)$ (where $\mcv$ is the set of vertices of $G$ and $\mce$ is the set of edges of $G$) which has $n$ vertices and $m$ edges. Undirected edge connecting vertices $u$ and $v$ will be denoted as $\edge{u,v}$ and directed edge -- as $\dir{uv}$.

Set of all neighbors (adjacent vertices) of vertex $x$ in graph $G$ will be denoted as $\mcn_G(x)$ or simply as $\mcn(x)$. 

A generalized distance function in graph $G$ will be denoted as $\dist(A,B)$ where $A$ and $B$ are  either subgraphs of $G$ or arbitrary sets of vertices/edges of $G$. It is defined as a minimum distance between vertices that belong to objects of set $A$ and vertices that belong to objects from set $B$. For instance, if we have edge $e=\edge{x,y}$ and vertex $z$ then $\dist(e, z) = \min(\mathrm{dist}(x,z), \mathrm{dist}(y,z))$. We will also sometimes use notation $\dist_A(B)$ for the same expression.

\vvv

\textit{Edge-weight} function on graph $G$ is a mapping $w: \mcv\times\mcv \rightarrow \mbr$ that is positive on pairs of adjacent vertices and zero otherwise. A most common example is a unit function, that is, $w(u,v)= 1$ iff $\edge{u,v}\in\mce$. Pair $(G, w)$ is called an \textit{edge-weighted graph} (or an \textit{electric network} if weights are treated as edge conductances) and it can be used to define a random walk on $G$ as a time-homogeneous \textit{Markov chain} with $G$ as its state space, and with transition probabilities set by formulas
$$
p_{xy} = \frac{w(x,y)}{\sum_{z\in\mcn(x)}w(x,z)}\quad \textrm{if } \edge{x,y}\in \mce\ ; \quad p_{xy} = 0\quad \textrm{if } \edge{x,y}\notin \mce
$$
for any $x, y\in \mcv$. This means that at each step of the walk we "choose" transition edge with probability proportional to the edge's weight.

See \cite{DoyleSnell}, \cite{AldousFill}, \cite{LyonsPeres} for more formal definition of random walks on graphs, Markov chains and electric networks.

Among random walks on finite graphs \textit{simple} or \textit{symmetric} random walk (defined by edge-weight function constant on $\mce$) is the most commonly used. For this walk, if vertex $x$ has degree $k$ then $p_{xy}= 1/k$ for any adjacent vertex $y$.

Most texts on graphs and probability only consider random walks which are constructed from an edge-weight function. In this article, however, we will sometimes use "generalized" random walks with absorbing vertex which are not, strictly speaking, random walks on graphs (such a walk is still, however, a finite time-homogeneous Markov chain). Namely, a Markov chain obtained from a random walk by declaring exactly one vertex $a$ as \textit{absorbing} and setting probabilities for all transitions out of $a$ according to formula $p_{ay} = \delta_{ay}$ ($\delta$ is Kronecker's delta function) will be called $\mca$-walk. In this case we will set $\mcn(a)=\emptyset$.

Any "standard" random walk is in fact a time-reversible \textit{Markov chain} (\cite{Doob}, \cite{AldousFill}) with graph $G$ being its diagram where each edge of $G$ represents two oppositely oriented transitions. This Markov chain is \textit{strongly connected}, meaning that you can reach any vertex $y$ from any other vertex $x$. An $\mca$-walk is "almost" strongly connected -- you can reach any vertex $y$ from any vertex $x$ unless $x=a$. 

We will call average (that is, expected value of) length of walk's paths that begin at some vertex $x$ and end as soon as they reach $a$, \textit{hitting time} for vertex $x$, and will denote it as $\mch_G(x, a)$ or simply $\mch(x, a)$. When it will not end in confusion we will use an even simpler notation $\mch_G(x)$ or $\mch(x)$. This is also sometimes called \textit{access time} or \textit{absorption time} since it is the expected value of the time it will take for the walker to reach (and therefore be "absorbed" by) vertex $a$. 

For most of the facts (theorems, lemmas etc.) about hitting time in this article we can switch between $\mca$-walks and "standard" random walks. Obviously, computing $\mch(x, a)$ for a "regular" random walk $R$ is the same as computing it for the $\mca$-walk obtained from $R$ by declaring vertex $a$ absorbing and changing probabilities for transitions out of $a$ accordingly.

Treating vertex $a$ as the fixed absorbing state will help us to simplify some proofs. Therefore we will often convert a given random walk into an $\mca$-walk and consider vertex $a$ an absorbing vertex, and in all figures in this article we will distinguish that vertex by drawing a circle around it. 

\vvv

\begin{mprop}
For any $\mca$-walk numbers $h_x = \mch(x)$ satisfy the following system of linear equations
\begin{align}
\begin{cases}
\label{eq:maineq_graph}
h_x - \displaystyle{\sum_{y \in \mcv}p_{xy}\cdot h_y} &= 1,\quad x \neq a \\
h_a                                                     &= 0
\end{cases}
\end{align}
\end{mprop}

\begin{proofP}

Let us consider matrix $\mathbf{L}$ with elements $L_{xy} = \delta_{xy} - (1-\delta_{xa})p_{xy}$. Then system \ref{eq:maineq_graph} is equivalent to $\mathbf{L}\cdot\mathbf{h} = \{1-\delta_{xa}\}$. In other words, matrix $\mathbf{L}$ maps vector $\mathbf{h} = \{h_x\}$ into vector $\{1, 1, 1, \ldots, 0\}$ where all vector coordinates are 1, except for the one indexed by $a$ which equals 0; for ease of notation we assume that $a$ is the last indexed vertex. This matrix $\mathbf{L}$ is a so-called \textit{normalized Laplacian matrix} of a random walk (see \cite{Spielman}).

Proof immediately follows from the fact that any path to $a$ that starts in $x$ begins with a transition to one of its neighbors $y$ with probability $p_{xy}$. Vertex $a$ is the only exception since there are no transitions out of $a$, and $\mch(a)=0$. 
\end{proofP}

\begin{mprop}
\label{prop:mnonsing}
For any $\mca$-walk on $G$ matrix $\mathbf{L}$ is nonsingular.
\end{mprop}

\begin{proofP}
We will use the fact that matrix $\mathbf{L}$ is diagonally dominant, meaning that in every row absolute value of its diagonal element is greater than or equal to the sum of absolute values of all non-diagonal elements. It is also strictly dominant (inequality mentioned just above is strict) in the last row. Now if $\mathbf{L} = \{L_{ij}\}$ is singular then there exists a non-zero vector $\pmb{x}$ such that $\mathbf{L}\pmb{x} = \pmb{0}$. Since $\pmb{x}$ is non-zero then maximum absolute value of its coordinates $|\pmb{x}|_\infty$ is positive. Assuming that maximum is reached at index $k$ we have $|x_k| = |x|_\infty$ and since $\sum L_{ki}x_i = 0$ it follows that
$$
\left|L_{kk}x_k\right| = \left|\sum_{i\neq k} L_{ki} x_i\right| \lte |x_k|\sum_{i\neq k} |L_{ki}|
\lte |x_k||L_{kk}|
$$

In order for this inequality chain to be valid all inequalities here must be, in fact, equalities. That means that for every index $i$ such that $L_{ki}\neq 0$ (that is, corresponding vertices of the graph are adjacent) we must have $|x_i| = |x_k|$ (and therefore, $x_i = |\pmb{x}|_\infty$). It follows then that all $|x_j|$ are the same because $G$ is connected, which is impossible because that would contradict any of the equations corresponding to a vertex adjacent to $a$.

This means, among other things, that system of linear equations (\ref{eq:maineq_graph}) always has exactly one solution.
\end{proofP}

\xxx

Below you will find a few simple graphs with values of $\mch$ printed next to the vertices (values are computed for simple random walks). Absorbing state vertex $a$ is, as always, circled.

\vvv
\begin{figure}[H]
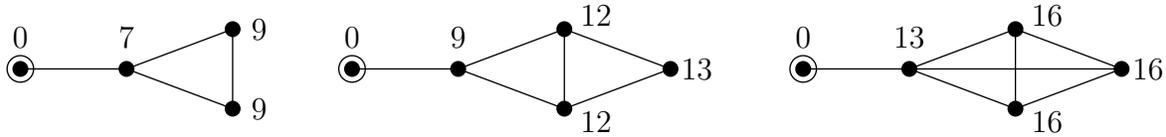

\begin{asy}
// ----------------- graph 1
// VERTICES
real x10 = -25; real y10 = -20;  fill(circle((x10,y10),3));  // 0
real x11 = 15;  real y11 = -20;  fill(circle((x11,y11),3));  // 1
real x12 = 55;  real y12 = -5;   fill(circle((x12,y12),3));  // 2
real x13 = 55;  real y13 = -35;  fill(circle((x13,y13),3));  // 3
draw(circle((x10,y10),5));

// EDGES
draw((x10, y10) -- (x11, y11)); // 0-1
draw((x11, y11) -- (x12, y12)); // 1-2
draw((x11, y11) -- (x13, y13)); // 1-3
draw((x12, y12) -- (x13, y13)); // 2-3

// LABELS
label("0", (x10, y10+12)); // 0
label("7", (x11, y11+12)); // 1
label("9", (x12+10, y12)); // 2
label("9", (x13+10, y13)); // 3

// ----------------- graph 2
// VERTICES
real x20 = 100; real y20 = -20; fill(circle((x20,y20),3)); // 0
real x21 = 140; real y21 = -20; fill(circle((x21,y21),3)); // 1
real x22 = 180; real y22 = -35; fill(circle((x22,y22),3)); // 2
real x23 = 180; real y23 = -5;  fill(circle((x23,y23),3)); // 3
real x24 = 220; real y24 = -20; fill(circle((x24,y24),3)); // 4
draw(circle((x20,y20),5));

// EDGES
draw((x20, y20) -- (x21, y21)); // 0-1
draw((x21, y21) -- (x22, y22)); // 1-2
draw((x21, y21) -- (x23, y23)); // 1-3
draw((x22, y22) -- (x23, y23)); // 2-3
draw((x22, y22) -- (x24, y24)); // 2-4
draw((x23, y23) -- (x24, y24)); // 3-4

// LABELS
label("0", (x20, y20+12));    // 0
label("9", (x21, y21+12));    // 1
label("12", (x22+12, y22-5)); // 2
label("12", (x23+12, y23+5)); // 3
label("13", (x24+10, y24));   // 4

// ----------------- graph 3
// VERTICES
real x30 = 270; real y30 = -20; fill(circle((x30,y30),3)); // 0
real x31 = 310; real y31 = -20; fill(circle((x31,y31),3)); // 1
real x32 = 350; real y32 = -35; fill(circle((x32,y32),3)); // 2
real x33 = 350; real y33 = -5;  fill(circle((x33,y33),3)); // 3
real x34 = 390; real y34 = -20; fill(circle((x34,y34),3)); // 4
draw(circle((x30,y30),5));

// EDGES
draw((x30, y30) -- (x31, y31)); // 0-1
draw((x31, y31) -- (x32, y32)); // 1-2
draw((x31, y31) -- (x33, y33)); // 1-3
draw((x31, y31) -- (x34, y34)); // 1-4
draw((x32, y32) -- (x33, y33)); // 2-3
draw((x32, y32) -- (x34, y34)); // 2-4
draw((x33, y33) -- (x34, y34)); // 3-4

// LABELS
label("0", (x30, y30+12));    // 0
label("13", (x31, y31+12));    // 1
label("16", (x32+12, y32-5)); // 2
label("16", (x33+12, y33+5)); // 3
label("16", (x34+10, y34));   // 4
\end{asy}
\caption{A few simple examples}
\label{fig:simple_ex}
\end{figure}
\vvv

Matrix $\mathbf{L}$ for the first graph in Fig.\ref{fig:simple_ex} looks like this
$$
\begin{bmatrix}
1 & -\frac13 & -\frac13 & -\frac13 \\
-\frac12 & 1 & -\frac12 & 0 \\
-\frac12 & -\frac12 & 1 & 0 \\
0 & 0 & 0 & 1 \\
\end{bmatrix}
$$

The examples above might lead you to suspect that the values of function $\mch$ are always integers in case of a simple random walk. This is true for a tree (we will prove that in the next section) but it is not so for an arbitrary connected graph. Here are some examples of graphs with non-integer values of $\mch$.

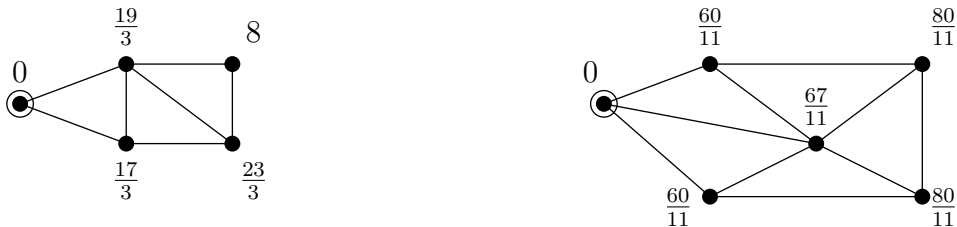
\vvv
\begin{figure}[H]
\begin{center}
\begin{asy}
label("0", (-25, -8));
fill(circle((-25,-20),3));
draw((-25, -20) -- (15, -5));
draw(circle((-25,-20),5));

label("$\frac{19}3$", (15, 9));
fill(circle((15,-5),3));

label("$\frac{17}3$", (15, -49));
fill(circle((15,-35),3));
draw((-25, -20) -- (15, -35));
draw((15, -35) -- (15, -5));
draw((15, -35) -- (55, -35));

label("8", (63, 7));
fill(circle((55,-5),3));
draw((15, -5) -- (55, -5));
label("$\frac{23}3$", (63, -49));
fill(circle((55,-35),3));
draw((15, -5) -- (55, -35));
draw((55, -5) -- (55, -35));

// -------------------------------

label("0", (190, -8));
fill(circle((195,-20),3));
draw((195, -20) -- (235, -5));
draw(circle((195,-20),5));

label("$\frac{60}{11}$", (235, 9));
fill(circle((235,-5),3));

label("$\frac{60}{11}$", (223, -59));
fill(circle((235,-55),3));
draw((195, -20) -- (235, -55));
draw((235, -55) -- (315, -55));

label("$\frac{80}{11}$", (323, 9));
fill(circle((315,-5),3));
draw((235, -5) -- (315, -5));
label("$\frac{80}{11}$", (323, -59));
fill(circle((315,-55),3));
draw((315, -5) -- (315, -55));

fill(circle((275,-35),3));
label("$\frac{67}{11}$", (275, -21));
draw((195, -20) -- (275, -35));
draw((235, -5) -- (275, -35));
draw((235, -55) -- (275, -35));
draw((315, -5) -- (275, -35));
draw((315, -55) -- (275, -35));

\end{asy}
\end{center}
\end{figure}
\vvv

\vvv
\begin{figure}[H]
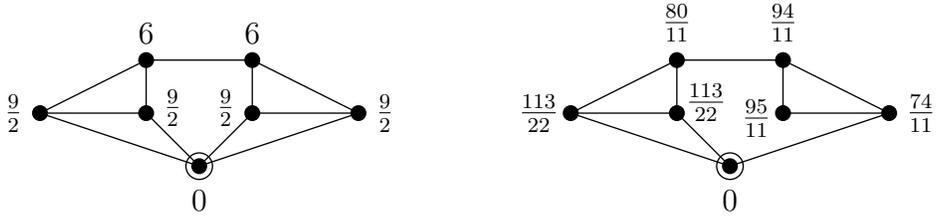

\begin{center}
\begin{asy}
fill(circle((40,0),3));
label("6", (40, 10));
fill(circle((80,0),3));
label("6", (80, 10));

fill(circle((0,-20),3));
label("$\frac92$", (-10, -20));
fill(circle((40,-20),3));
label("$\frac92$", (50, -18));
fill(circle((80,-20),3));
label("$\frac92$", (70, -18));
fill(circle((120,-20),3));
label("$\frac92$", (130, -20));

fill(circle((60,-40),3));
label("0", (60, -53));
draw(circle((60,-40),5));

draw((40, 0) -- (80, 0));

draw((40, 0) -- (0, -20));
draw((40, 0) -- (40, -20));
draw((80, 0) -- (80, -20));
draw((80, 0) -- (120, -20));

draw((0, -20) -- (40, -20));
draw((80, -20) -- (120, -20));

draw((60, -40) -- (0, -20));
draw((60, -40) -- (40, -20));
draw((60, -40) -- (80, -20));
draw((60, -40) -- (120, -20));

// -------------------------------

fill(circle((240,0),3));
label("$\frac{80}{11}$", (240, 14));
fill(circle((280,0),3));
label("$\frac{94}{11}$", (280, 14));

fill(circle((200,-20),3));
label("$\frac{113}{22}$", (188, -20));
fill(circle((240,-20),3));
label("$\frac{113}{22}$", (251, -16));
fill(circle((280,-20),3));
label("$\frac{95}{11}$", (270, -22));
fill(circle((320,-20),3));
label("$\frac{74}{11}$", (332, -20));

fill(circle((260,-40),3));
label("0", (260, -53));
draw(circle((260,-40),5));

draw((240, 0) -- (280, 0));

draw((240, 0) -- (200, -20));
draw((240, 0) -- (240, -20));
draw((280, 0) -- (280, -20));
draw((280, 0) -- (320, -20));

draw((200, -20) -- (240, -20));
draw((280, -20) -- (320, -20));

draw((260, -40) -- (200, -20));
draw((260, -40) -- (240, -20));

draw((260, -40) -- (320, -20));

\end{asy}
\caption{Examples with non-integer hitting times}
\end{center}
\end{figure}
\vvv



\msection{Simple random walks on connected graphs}

Consider a "straight-line" tree graph, which is more formally known as $P_m$ --  \textit{path graph} or \textit{path tree of length} $m$ -- with $m+1$ vertices and $m$ edges pictured below.

\vvv
\begin{figure}[H]
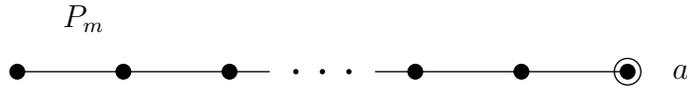

\begin{center}
\begin{asy}
label("$P_m$", (0,0));
fill(circle((-25,-20),3));
draw((-25, -20) -- (15, -20));
fill(circle((15,-20),3));
draw((15, -20) -- (55, -20));
fill(circle((55,-20),3));
draw((55, -20) -- (70, -20));
fill(circle((80,-20),1));
fill(circle((90,-20),1));
fill(circle((100,-20),1));

draw((110, -20) -- (125, -20));
fill(circle((125,-20),3));
draw((125, -20) -- (165, -20));
fill(circle((165,-20),3));
draw((165, -20) -- (205, -20));
fill(circle((205,-20),3));
label("$a$", (225,-20));
draw(circle((205,-20),5));

\end{asy}
\caption{Path tree}
\end{center}
\end{figure}
\vvv

If we index vertices from right to left starting at zero ($a_0$ = $a$) and denote hitting time $\mch(a_k)$ for vertex $a_k$ as $h_k$ then system (\ref{eq:maineq_graph}) turns into the following recurrence equation
\begin{align}
\label{eq:maineq_path}
h_k = 1 + (h_{k-1}+h_{k+1})/2,\quad 0 < k < m
\end{align}
Also, obviously $h_0 = 0$ and $h_m = 1 + h_{m-1}$.

\vvv
\begin{figure}[H]
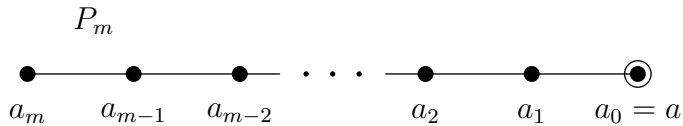

\begin{center}
\begin{asy}
label("$P_m$", (0,0));
fill(circle((-25,-20),3));
draw((-25, -20) -- (15, -20));
fill(circle((15,-20),3));
draw((15, -20) -- (55, -20));
fill(circle((55,-20),3));
draw((55, -20) -- (70, -20));
fill(circle((80,-20),1));
fill(circle((90,-20),1));
fill(circle((100,-20),1));

draw((110, -20) -- (125, -20));
fill(circle((125,-20),3));
draw((125, -20) -- (165, -20));
fill(circle((165,-20),3));
draw((165, -20) -- (205, -20));
fill(circle((205,-20),3));
draw(circle((205,-20),5));

label("$a_0 = a$", (205,-35));
label("$a_1$", (165,-35));
label("$a_2$", (125,-35));

label("$a_m$", (-25,-35));
label("$a_{m-1}$", (15,-35));
label("$a_{m-2}$", (55,-35));

\end{asy}
\caption{Indexed path tree}
\label{fig:indx_path_tree}
\end{center}
\end{figure}
\vvv

Now let's rewrite the equation (\ref{eq:maineq_path}) as
\begin{align}
\label{eq:maineq_path_delta}
h_{k+1}-h_k = 2 + (h_{k-1}-h_k),\quad 0 < k < n
\end{align}

If we denote difference $h_k-h_{k-1}$ as $d_k$ then $d_m = 1$ and therefore $d_{m-1} = 3$, $d_{m-2} = 5$ etc., $d_{m-k} = 2k + 1$. It follows immediately that $h_m - h_{m-k} = k^2$ and since $h_0 = 0$ we obtain that $h_m = m^2$ and consequently, $h_{m-k} = m^2-k^2$. This is an easy and well-known fact (see, for instance, \cite{BlomHolSan}). This gives us a very simple upper bound for $\mch_R(G)$ for the path trees.

Using a somewhat similar technique we can prove the same for any tree $T$.

\begin{mthm}
\label{thm:st1}
For simple random walk on any finite tree $T$ with $m$ edges inequality $\mch(T) \lte m^2$ holds true.
\end{mthm}

\begin{proof}
We need to prove inequality $\mch(x,a) \lte m^2$ for any pair of vertices in $T$. To do that we will fix vertex $a$, declare it absorbing and convert the given walk into an $\mca$-walk. Also without any loss of generality we can consider only vertices from connected component of $T\backslash a$ that contains $x$.

We will call $\mcl(x) = T\backslash T_x$ the \textit{tail} of vertex $x$, where $T_x$ is the set of all vertices that lie in the component of $T\backslash x$ that contains $a$. Simply put, tail of $x$ is the sub-tree defined by all the vertices that cannot be connected to $a$ by a path bypassing $x$. We will denote by $\ell(x)$ the number of edges in sub-tree $\mcl(x)$. In the following picture sub-tree $\mcl(x)$ is shown with thicker lines; $\ell(x) = 9$.

\vvv
\begin{figure}[H]
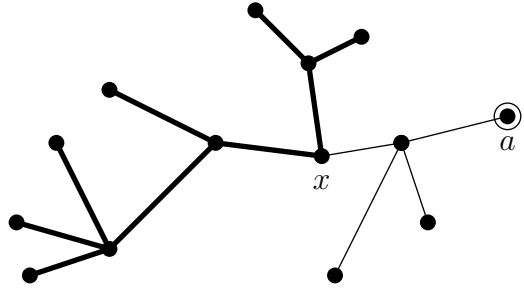

\begin{center}
\begin{asy}
fill(circle((-25,-20),3));
fill(circle((-25,-80),3));
draw((-25, -20) -- (15, -40), linewidth(2pt));
draw((-25, -80) -- (15, -40), linewidth(2pt));

draw((-25, -80) -- (-45, -40), linewidth(2pt));
fill(circle((-45,-40),3));
draw((-25, -80) -- (-60, -70), linewidth(2pt));
fill(circle((-60,-70),3));
draw((-25, -80) -- (-55, -90), linewidth(2pt));
fill(circle((-55,-90),3));

fill(circle((15,-40),3));
draw((15, -40) -- (55, -45), linewidth(2pt));
fill(circle((55,-45),3));

label("$x$", (55,-55));

draw((55, -45) -- (50, -10), linewidth(2pt));
fill(circle((50,-10),3));
draw((50, -10) -- (30, 10), linewidth(2pt));
fill(circle((30,10),3));
draw((50, -10) -- (70, 0), linewidth(2pt));
fill(circle((70,0),3));

fill(circle((85,-40),3));
draw((55, -45) -- (85, -40));

fill(circle((85,-40),3));
draw((85, -40) -- (95, -70));
fill(circle((95,-70),3));
draw((85, -40) -- (60, -90));
fill(circle((60,-90),3));

draw((85, -40) -- (125, -30));
fill(circle((125,-30),3));
draw(circle((125,-30),5));

label("$a$", (125,-40));

\end{asy}
\caption{"Tail of a vertex" example}
\end{center}
\end{figure}
\vvv

\begin{mlemma}
\label{lem:tail_eq}
For any vertex $x$ and any of its neighbors $y\in\mcl(x)$, the following equality holds true
\begin{align}
\label{eq:tail_eq}
\mch(y) - \mch(x) = 2 \ell(y) + 1
\end{align}
\end{mlemma}

\begin{proofL} We will prove that using induction by $\ell(x)$. Basis: $\ell(x) = 1$, that is, $y$ is a leaf (pendant vertex). Then obviously, $\mch(y) - \mch(x) = 1$ and $\ell(y) = 0$ which proves the basis. Now, for the step of induction we will take some pair of $x$ and $y$. Since $\ell(y) < \ell(x)$ then lemma is true for the pair of $y$ and any of its other neighbors in $T$ -- we will name them $z_i, i = 1, ..., k$. 

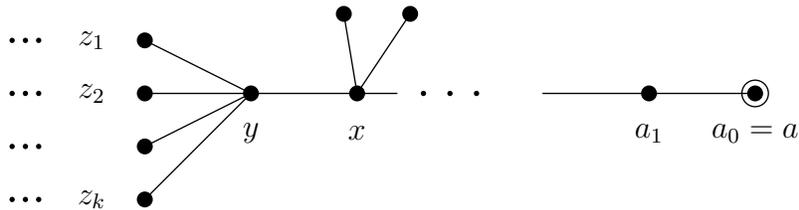
\vvv
\begin{figure}[H]
\begin{center}
\begin{asy}
fill(circle((-75,-20),1));
fill(circle((-70,-20),1));
fill(circle((-65,-20),1));
fill(circle((-75,-40),1));
fill(circle((-70,-40),1));
fill(circle((-65,-40),1));
fill(circle((-75,-60),1));
fill(circle((-70,-60),1));
fill(circle((-65,-60),1));
fill(circle((-75,-80),1));
fill(circle((-70,-80),1));
fill(circle((-65,-80),1));

fill(circle((-25,-20),3));
fill(circle((-25,-40),3));
fill(circle((-25,-60),3));
fill(circle((-25,-80),3));

draw((-25, -20) -- (15, -40));
label("$z_1$", (-45,-20));
draw((-25, -40) -- (15, -40));
label("$z_2$", (-45,-40));
draw((-25, -60) -- (15, -40));

draw((-25, -80) -- (15, -40));
label("$z_k$", (-45,-80));

fill(circle((15,-40),3));
draw((15, -40) -- (55, -40));
fill(circle((55,-40),3));
draw((55, -40) -- (70, -40));

fill(circle((80,-40),1));
fill(circle((90,-40),1));
fill(circle((100,-40),1));
label("$y$", (15,-55));
label("$x$", (55,-55));

draw((55, -40) -- (75, -10));
fill(circle((75,-10),3));
draw((55, -40) -- (50, -10));
fill(circle((50,-10),3));

draw((125, -40) -- (165, -40));
fill(circle((165,-40),3));
draw((165, -40) -- (205, -40));
fill(circle((205,-40),3));
draw(circle((205,-40),5));

label("$a_0 = a$", (205,-55));
label("$a_1$", (165,-55));
\end{asy}
\end{center}
\end{figure}
\vvv

Writing out the equation (\ref{eq:maineq_graph}) for vertex $y$ we get
\begin{align}
\label{eq:vw_induction}
\mch(y) = 1 + \frac1{k+1}\left(\mch(x)+\sum_{i=1}^k \mch(z_i)\right)
\end{align}
or
\begin{align}
(k+1)\mch(y) = k+1 + \mch(x)+\sum_{i=1}^k \mch(z_i)\nn
\end{align}

Now we rewrite that using equation (\ref{eq:tail_eq}) for each pair $(y, z_i)$ and equality $\ell(y) = k + \sum_{i=1}^k \ell(z_i)$.
\begin{align}
(k+1)\mch(y) &= k+1 + \mch(x)+k\mch(y)+k+2\sum_{i=1}^k \ell(z_i)\nncr
\mch(y) &= 2k+1 + \mch(x)+2\sum_{i=1}^k \ell(z_i)\nn
\end{align}
which gives us
$$
\mch(y) - \mch(x) = 2\ell(y) + 1.
$$
\end{proofL}

Now connect vertices $a$ and $x$ by a non-self-intersecting path of length $d$ (there is only one such path) and index its vertices in exactly the same way we did above in the Fig.\ref{fig:indx_path_tree}.

Then we have set of equations $\mch(a_k) - \mch(a_{k-1}) = 2\ell(a_k) + 1$ and summing them up we get $\mch(x) - \mch(a) = 2\sum \ell(a_k) + d$. Since $\mch(a) = 0$ and $\ell(a_k) \lte \ell(a_{k-1})-1$, it follows that
\begin{align}
\mch(x) & \lte 2d\cdot \ell(a_1) + d - 2\sum_{i=0}^{d-1} i = 2d\cdot \ell(a) - d^2 \nncr
        & = \ell(a)^2 - (\ell(a)-d)^2 \lte \ell(a)^2 \lte m^2 \nn
\end{align}
and $\ell(a)^2$ is equal to $m^2$ if and only if there is only one connected component in $T\backslash a$. 
\end{proof}

\vspace{10pt}

Below are a few corollaries of Lemma \ref{lem:tail_eq}. The first three are nearly self-evident and we will leave them as easy exercises for the reader.

\begin{mcrl}
For simple random walk on a tree all hitting times are integers.
\end{mcrl}

\begin{mcrl}
$\mch(x) \lte \ell(a)^2-\ell(x)^2$.
\end{mcrl}

\begin{mcrl}
$\mch(x) \lte m^2-(m-d)^2$, where $d = \dist(a,x)$.
\end{mcrl}

Another widely used characteristic of a random walk on graph $G$ is \textit{commute time} $\mcc(x,y)$ between vertices $x, y \in G$. It is defined as sum of hitting times $\mch(x,y)+\mch(x,y)$ and is equal to the expected (average) time it takes the walker to "travel" from $x$ to $y$ and then back to $x$. Imagine that $x$ is your house and $y$ is your office -- that immediately explains the name, doesn't it?

\begin{mcrl}
If $x$ and $y$ are two vertices in $T$ then
$$
\mcc(x,y) = 2m\cdot\dist(x,y)
$$
For instance, commute time between any two neighboring vertices equals twice the number of the edges.
\end{mcrl}

\begin{proofL}
Again, connect $x$ and $y$ with path $\Pi = [y_0y_1\ldots y_d]$ of length $d=\dist(x,y)$ where $y_0 = y$ and $y_d = x$ and for every index $0 \lte k \lte d$ denote by $e_k$ the number of edges in the component of subtree $(T\backslash\Pi)\cup y_k$ that contains $y_k$. Visually, if you "take" vertices $x$ and $y$ in your right and left hands so that path $\Pi$ becomes a horizontal rope connecting your hands and shake the entire tree so that the other vertices and edges will drop down to dangle from the nodes of $\Pi$, then $e_k$ represents the number of edges "hanging" on the $k$-th node $y_k$.

\vvv
\begin{figure}[H]
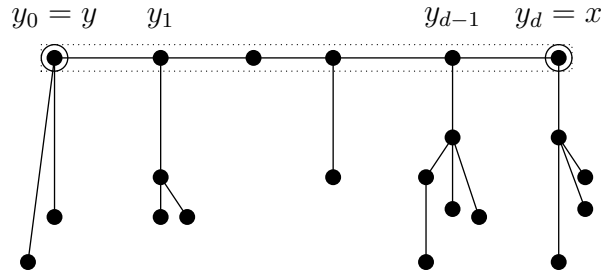

\begin{center}
\begin{asy}
// variables
real y0 = 0;
real y1 = -30, y2 = -45, y3 = -57, y4 = -60, y5 = -77;;
real x0 = 15, x1 = 55, x2 = 90, x3 = 120, x4 = 165, x5 = 205;
real dx = 30, ddx = 10, dy = 15;
// path
fill(circle((x0,y0),3));
draw((x0, y0) -- (x1, y0));
fill(circle((x1,y0),3));
draw((x1, y0) -- (x2, y0));
fill(circle((x2,y0),3));
draw((x2, y0) -- (x3, y0));
fill(circle((x3,y0),3));
draw((x3, y0) -- (x4, y0));
fill(circle((x4,y0),3));
draw((x4, y0) -- (x5, y0));
fill(circle((x5,y0),3));
// circling the path's ends
draw(circle((x0,y0),5));
draw(circle((x5,y0),5));
//dangles
// #0
draw((x0, y0) -- (x0, y4));
fill(circle((x0,y4),3));
draw((x0, y0) -- (x0-ddx, y5));
fill(circle((x0-ddx,y5),3));
// #1
draw((x1, y0) -- (x1, y2));
fill(circle((x1,y2),3));
draw((x1, y2) -- (x1, y4));
fill(circle((x1,y4),3));
draw((x1, y2) -- (x1+ddx, y4));
fill(circle((x1+ddx,y4),3));
// #2
// #3
draw((x3, y0) -- (x3, y2));
fill(circle((x3,y2),3));
// #4
draw((x4, y0) -- (x4, y1));
fill(circle((x4,y1),3));
draw((x4, y1) -- (x4-ddx, y2));
fill(circle((x4-ddx,y2),3));
draw((x4, y1) -- (x4, y3));
fill(circle((x4,y3),3));
draw((x4, y1) -- (x4+ddx, y4));
fill(circle((x4+ddx,y4),3));
draw((x4-ddx, y2) -- (x4-ddx, y5));
fill(circle((x4-ddx,y5),3));
// #5
draw((x5, y0) -- (x5, y1));
fill(circle((x5,y1),3));
draw((x5, y1) -- (x5+ddx, y2));
fill(circle((x5+ddx,y2),3));
draw((x5, y1) -- (x5+ddx, y3));
fill(circle((x5+ddx,y3),3));
draw((x5, y1) -- (x5, y5));
fill(circle((x5,y5),3));
// labels
label("$y_0 = y$", (x0,y0+dy));
label("$y_1$",     (x1,y0+dy));
label("$y_{d-1}$", (x4,y0+dy));
label("$y_d = x$", (x5,y0+dy));
// "box" around path
real dp = 5;
draw((x0-dp,y0-dp)--(x5+dp,y0-dp)--(x5+dp,y0+dp)--(x0-dp,y0+dp)--(x0-dp,y0-dp), dotted);
\end{asy}
\caption{$\{e_k\} = {2,3,0,1,5,4}$}
\end{center}
\end{figure}
\vvv

Let's assume that vertex $y$ is absorbing. Then from Lemma \ref{lem:tail_eq}
$$
\mch(y_k, y) - \mch(y_{k-1}, y) = 2\ell(y_k) + 1
$$
and after expressing $\ell(y_k)$ through $\{e_i\}$
$$
\mch(y_k, y) - \mch(y_{k-1}, y) = 2\left(e_k+e_{k+1}+\ldots+e_d+(d-k)\right) + 1
$$
Adding these formulas for $k=1,\ldots, d$ we obtain
\begin{align}
\mch(x,y) &= \mch(y_d, y) - \mch(y_0, y) \nncr
          &= 2\sum_{k=1}^d k\cdot e_k + \sum_{k=1}^d (2(d-k)) + 1) \nncr
          & = 2\sum_{k=0}^d k\cdot e_k + d^2 \nn
\end{align}

Now swap $x$ and $y$, reverse indexing, and we have $\mch(y, x) = 2\sum_{k=0}^d (d-k)\cdot e_k + d^2$. Add these two equations together and we get
\begin{align}
\mcc(x,y) &= \mch(x,y) + \mch(y,x) \nncr
   &= 2d^2 + 2\sum_{k=0}^d d\cdot e_k = 2d^2 + 2d\sum_{k=0}^d e_k = 2d^2 + 2d(m-d) = 2md \nn
\end{align}
\end{proofL}

I should mention here that this corollary also immediately follows from one of the theorems in article \cite{Ruzzo} which we will use later (as an example of electric network approach) to prove our Theorem \ref{thm:asg1}. Same theorem (together with monotonicity laws for electric resistance) proves that for the case of finite connected graph we always have inequality $\mcc(x,y) \lte 2md$. The theorem itself and this inequality can also be proved in a rather straightforward way using \textit{harmonic functions} approach (see \cite{HopKan} or \cite{LyonsPeres}).

\xxx

Now let us move on to the connected graphs in general.

\begin{mthm}
\label{thm:sg1}
For simple random walk on graph $G = (\mcv, \mce)$ inequality
$$
\mch(x,y) \lte m^2 - (m-d)^2, 
$$
where $d = \dist(x,y)$, holds true for any two vertices $x,y \in \mcv$.
\end{mthm}

\begin{proof}
Again, to simplify the notation and reasoning, we will -- just as in \ref{thm:st1} -- rename vertex $y$ to $a$, declare it absorbing and convert the given walk into an $\mca$-walk.

\begin{mlemma}
\label{lem:sg1_1}
Consider edge $e=\edge{a,x}\in\mce$ such that vertices $a$ and $x$ are in the same connected component of graph $G' = G\backslash e$. Then $\mch_G(x) \lte \mch_{G'}(x)$. In other words, removing such edge cannot decrease hitting time.
\end{mlemma}

\begin{proofL}
Seems self-evident -- erasing an edge leading directly into the absorbing state should only increase hitting time. However this is not an entirely trivial fact. Let us denote $k = \deg_{G'}(x)$. If $k=0$ then $x$ is only connected to $a$ and there is nothing to prove. 

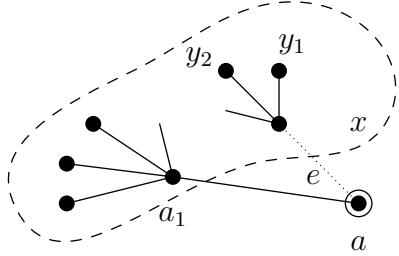
\vvv
\begin{figure}[H]
\begin{center}
\begin{asy}
fill(circle((125,-50),3));
fill(circle((125,-35),3));
fill(circle((135,-20),3));

draw((125, -50) -- (165, -40));
draw((125, -35) -- (165, -40));
draw((135, -20) -- (165, -40));

fill(circle((165,-40),3));
label("$a_1$", (165,-55));

draw((165, -40) -- (235, -50));
draw((165, -40) -- (160, -20));

fill(circle((235,-50),3));
draw(circle((235,-50),5));

label("$a$", (235,-65));

fill(circle((205,-20),3));
label("$x$", (235,-20));

draw((235, -50) -- (205, -20), dotted);
label("$e$", (218, -40));

fill(circle((205,0),3));
draw((205, -20) -- (205, 0));
label("$y_1$", (210,10));

fill(circle((185,0),3));
draw((205, -20) -- (185, 0));
label("$y_2$", (175,5));

draw((205, -20) -- (185, -15));

pair[] z={(105,-50), (150,0), (200,25), (245,-20), (195,-35)};
draw(z[0]..z[1]..z[2]..z[3]..z[4]..cycle, dashed);

\end{asy}
\end{center}
\end{figure}
\vvv

Now consider $\partial_v = \mch_{G'}(v) - \mch_G(v)$ for any $v\in\mcv$. We need to show that all these numbers are non-negative (they are actually positive). Subtracting systems (\ref{eq:maineq_graph}) for $G$ and $G'$ from each other we get the same matrix $\mathbf{L}$ on the left side of the resulting system but the right side vector is different from (\ref{eq:maineq_graph}). Its coordinates are zero for all indices (vertices) with exception of $x$. We have
\begin{align}
\begin{cases}
\partial_v - \frac1{\deg(v)}\displaystyle\sum_{z\in \mcn(v)} \partial_z &=  0, \quad v \neq x \\
\partial_x - \frac1{k+1}\displaystyle\sum_{z\in \mcn(x)} \partial_z &=  \frac1{k(k+1)}\displaystyle\sum_{z\in \mcn(x)} \mch_{G'}(z) \label{eq:diffAX}
\end{cases}
\end{align}
and that number on the right side of the second equation in (\ref{eq:diffAX}) is obviously positive. If we denote that number by $r$ then we have that matrix $\mathbf{L}$ maps vector $\{\partial_v\}$ to vector $\{r\delta_{vx}\}$. In all these formulas we are using $\mcn_G$ since it is easy to see that changing it to $\mcn_{G'}$ in cases where it is called for makes no difference.

Let us consider another set of numbers that satisfy similar system of equations. Namely, for each vertex $v$ define $s_v$ as the probability that our random walk starting in $v$ will be absorbed in $a$ with edge $e$ being its last transition.

Obviously for any vertex except $x$ you have
$$
s_v = \frac1{\deg(v)}\sum_{z\in \mcn(v)} s_z
$$
because each path starting in $v$ first goes to one of its neighbors (this transition not being edge $e$!) and then has to get from there to $a$ passing through $e$ in the end. For vertex $x$ 
$$
s_x = \frac1{k+1}\sum_{z\in \mcn(x)} s_z + \frac1{k+1}
$$

Thus matrix $\mathbf{L}$ maps vector $\{s_v\}$ into vector $\{\frac1{k+1}\delta_{vx}\}$. It follows then that vector $\{\partial_v\}$ equals to $r(k+1)\{s_v\}$ since they are both mapped to the same vector $\{r\delta_{vx}\}$ by nonsingular matrix $\mathbf{L}$ (see Proposition \ref{prop:mnonsing}).

Since $r$ is positive, and all the numbers $s_v$ are positive (graph $G'$ is connected) this concludes the proof of the lemma.
\end{proofL}

\vspace{5pt}

\begin{mlemma}
\label{lem:sg1_2}
$$
\sum_{z\in \mcn(a)} \mch(z) = 2m-\deg(a)
$$
\end{mlemma}

\begin{proofL}
For every edge $\edge{x,y}$ in graph $G$ let us represent it as two directed edges with different orientations $\dir{xy}$ and $\dir{yx}$. We will mark each directed edge $\dir{xy}$ with difference $d_{xy} = \mch(x)-\mch(y)$. Sum of all these numbers is obviously zero.

However, if we group them by start vertex $x$, then for every such group except for the case of $x = a$ sum of the numbers will be, by equation (\ref{eq:maineq_graph}), equal to $\deg(x)$. For $a$ the sum is $(-\sum_{z\in \mcn(a)} \mch(z))$. It follows then (since the sum of degrees of all vertices in a graph equals twice the number of edges)
$$
0 = \sum_{xy\in \mce}d_{xy} = \sum_{z \neq a} \deg(z) - \sum_{z\in\mcn(a)} \mch(z) = 2m - \deg(a) - \sum_{z\in\mcn(a)} \mch(z)
$$
which is exactly what we need.
\end{proofL}

Now using the lemmas above we will prove Theorem \ref{thm:sg1} by induction by the number of edges $m$. Basis of induction is obvious. Now let us connect $x$ to $a$ by a shortest path; its last edge will connect some vertex $a_1$ and $a$. If $a$ has any other incident edge $e$ besides $\edge{a,a_1}$ then by removing it and discarding components of connectedness that do not contain $x$ we will reduce $G$ to graph $G'$ with fewer edges than $G$. Thus, from Lemma \ref{lem:sg1_1} and induction hypothesis,
$$
\mch_G(x) \lte \mch_{G'}(x) \lte (m-1)^2 - ((m-1)-d)^2 = 2(m-1)d-d^2 < 2md - d^2 = m^2-(m-d)^2,
$$
which proves this case. If no such edge exists then $a$ is a pendant vertex with only one incident edge $\edge{a,a_1}$.

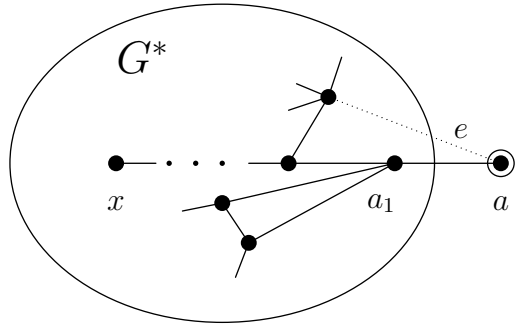
\vvv
\begin{figure}[H]
\begin{center}
\begin{asy}
// ELLIPSIS
	real xp = 80; real yp = -20;
fill(circle((xp,yp),1));
fill(circle((xp+10,-20),1));
fill(circle((xp+20,-20),1));

// VERTICES
	real x0 = xp+45; real y0 = yp;
	real x1 = xp+60; real y1 = yp+25;
	real x2 = xp+20; real y2 = yp-15;
	real x3 = xp+85; real y3 = yp;
	real x4 = xp+125; real y4 = yp;
	real x5 = xp-20; real y5 = yp;
	real x6 = xp+30; real y6 = yp-30;
fill(circle((x0, y0),3));
fill(circle((x1, y1),3));
fill(circle((x2, y2),3));
fill(circle((x3, y3),3));
fill(circle((x4, y4),3));
fill(circle((x5, y5),3));
fill(circle((x6, y6),3));
draw(circle((x4, y4),5));

// ELLIPSE
draw(ellipse((xp+20, yp), 80, 60));
label("\Large{$G^*$}", (xp-10, yp+40));

// EDGES
draw((x0-15, y0) -- (x0, y0));
// draw((x1, y1) -- (x3, y3));
draw((x2, y2) -- (x3, y3));
draw((x0, y0) -- (x3, y3));
draw((x3, y3) -- (x4, y4));
draw((x5, y5) -- (x5+15, y5));
draw((x2, y2) -- (x6, y6));
draw((x3, y3) -- (x6, y6));
draw((x4, y4) -- (x1, y1), dotted);
draw((x1, y1) -- (x0, y0));

draw((x6, y6) -- (x6-5, y6-13));
draw((x2, y2) -- (x2-15, y2-3));
draw((x1, y1) -- (x1-12, y1+5));
draw((x1, y1) -- (x1-15, y1-5));
draw((x1, y1) -- (x1+5, y1+15));

// LABELS
label("$a_1$", (x3-5, y3-15));
label("$a$", (x4, y4-15));
label("$x$", (x5, y5-15));
label("$e$", (x4-15, y4+12));

\end{asy}
\end{center}
\end{figure}
\vvv

Now consider graph $G^*=G'\backslash a$ and assign $a_1$ as its absorbing state (vertex) to create an $\mca$-walk. From Proposition \ref{prop:mnonsing} it follows that values of $\mch_G(v)$ are the same as solutions of system of linear equations (\ref{eq:maineq_graph}), therefore solution for $G^*$ is obviously the same as solution for $G$ restricted to $G^*$, from which $\mch_G(a_1)$ is subtracted.

From Lemma \ref{lem:sg1_2} we know that $\mch_G(a_1) = 2m-1$. Since we know that $\mch_{G^*}(x,a_1) \lte (m-1)^2 - ((m-1)-(d-1))^2$ and also $\mch_G(x,a) = \mch_{G^*}(x,a_1) + \mch_G(a_1, a)$, it immediately follows that $\mch_G(x,a) \lte (m-1)^2 - (m-d)^2 + 2m-1 = m^2 - (m-d)^2$.
\end{proof}

The following theorem is an obvious corollary of the last one.

\begin{mthm}
\label{thm:sg1c}
For a simple random walk on $G$ we have $\mch(G) \lte m^2$, with equality reached if and only if $G$ is path graph $P_m$.
\end{mthm}

\xxx

Now we will try to generalize this fact for the graphs endowed with a so called edge-cost function. 

It is often necessary to consider a case where edges of graph have different "lengths", or where "time" to transition along an edge is not constant (originally we assumed it is always equal to 1). Formally, each edge $e=\edge{x,y}$ can be assigned some (usually non-negative) \textit{cost} $f(e)$ which is associated with traveling (transitioning) along $e$. Cost $\mcc^f(\Pi)$ of any finite path $\Pi = \left[x_1x_2\ldots x_k\right]$ is determined as the sum of costs of all edges (transitions) in that path, that is $\mcc^f(\Pi) = \sum_{i=1}^{k-1} f(\edge{x_i,x_{i+1}})$.

For instance, if $f\equiv 1$ then path's cost is simply its length. In the example below we show cost function $f$ presented as numbers written next to the edges, and path $\Pi = [qrts]$ (shown with thicker lines) with cost $\mcc^f(\Pi) = f(\edge{q,r})+f(\edge{r,t})+f(\edge{t,s})=2+5+1 = 8$.

\vvv
\begin{figure}[H]
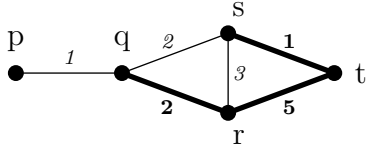

\begin{center}
\begin{asy}
// VERTICES
real x00 = 100; real y00 = -20; fill(circle((x00,y00),3)); // 0
real x01 = 140; real y01 = -20; fill(circle((x01,y01),3)); // 1
real x02 = 180; real y02 = -35; fill(circle((x02,y02),3)); // 2
real x03 = 180; real y03 = -5;  fill(circle((x03,y03),3)); // 3
real x04 = 220; real y04 = -20; fill(circle((x04,y04),3)); // 4

// EDGES
draw((x00, y00) -- (x01, y01)); // 0-1
draw((x01, y01) -- (x02, y02), linewidth(2pt)); // 1-2
draw((x01, y01) -- (x03, y03)); // 1-3
draw((x02, y02) -- (x03, y03)); // 2-3
draw((x02, y02) -- (x04, y04), linewidth(2pt)); // 2-4
draw((x03, y03) -- (x04, y04), linewidth(2pt)); // 3-4

// LABELS
label("p", (x00, y00+12));    // 0
label("q", (x01, y01+12));    // 1
label("r", (x02+4, y02-9)); // 2
label("s", (x03+4, y03+9)); // 3
label("t", (x04+10, y04));   // 4

label("{\scriptsize\it 1}", (x00+20, y00+6));    // 0-1
label("{\scriptsize\it 2}", (x01+17, y01+12));   // 1-2
label("{\scriptsize\bf 2}", (x01+17, y01-13));   // 1-3
label("{\scriptsize\it 3}", (x02+4,  y01));      // 2-3
label("{\scriptsize\bf 1}", (x02+23, y01+12));   // 2-4
label("{\scriptsize\bf 5}", (x02+23, y01-13));   // 3-4

\end{asy}
\caption{"Cost of path" example}
\end{center}
\end{figure}
\vvv

In real-life computational problems this is a very common occurrence. Time to transition (travel) along an edge (or some cost associated with that transition) is often non-constant and it has to be taken into consideration when computing total time (or some other type of "expense") to travel from one point to another.

Access (hitting) time $\mch^f(x, y) = \mch^f_G(x, y)$ relative to cost function $f$ is defined as expected value of cost function for random walk's path that starts in $x$ and stops when it reaches (is absorbed by) vertex $y$. Just as before, $\mch^f(G)$ is defined as $\max_{x,y\in G}\mch^f(x,y)$.

\begin{mthm}
\label{thm:fg1}
For a simple random walk on $G$, for any non-negative edge-cost function $f$ and any vertices $x,y\in\mcv$ inequality
$$
\mch^f(x,y)\ \lte\ \sum_{e\in \mce}\left(2\dist_y(e)+1\vphantom{\mch^f}\right)f(e)\ \lte\ m^2 \max(f)
$$
holds true.
\end{mthm}

\begin{proof}

As before, we will rename $y$ to $a$ and make it absorbing, converting regular random walk to an $\mca$-walk.

Second inequality can be left to the reader as an easy exercise (incidentally, the inequality between the first and the last expressions immediately follows from Theorem \ref{thm:sg1c}). We only need to prove the first inequality. It would seem we can simply reuse the proof of Theorem \ref{thm:sg1} using function $\deg^f(x) = \sum_{y\in\mcn(x)}f(\edge{x,y})$ instead of $\deg(x)$, and system
\begin{align}
\begin{cases}
\label{eq:maineq_graphf}
h_x - \displaystyle{\sum_{z \in \mcn(x)}p_{xz}\cdot h_z} &= \displaystyle{\sum_{z \in \mcn(x)}p_{xz}f(xz)},\quad x \neq a \\
h_a                                                    &= \quad 0
\end{cases}
\end{align}
instead of \ref{eq:maineq_graph}. Alas, that is not possible (at least not in the most direct manner) because Lemma \ref{lem:sg1_1} is not valid for arbitrary cost function $f(e)$. As a very simple example consider the following graph $G$ with graph $G'$ produced by erasing edge $\edge{a,c}$ from $G$.

\vvv
\begin{figure}[H]
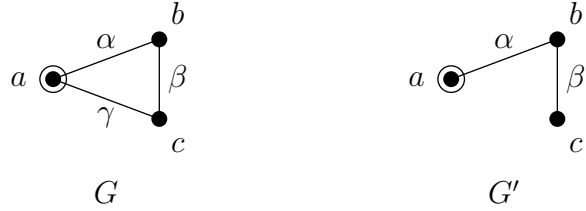

\begin{center}
\begin{asy}
// --- graph 1
// VERTICES
real x11 = 15;  real y11 = -20;  fill(circle((x11,y11),3));  // 1
real x12 = 55;  real y12 = -5;   fill(circle((x12,y12),3));  // 2
real x13 = 55;  real y13 = -35;  fill(circle((x13,y13),3));  // 3
draw(circle((x11,y11),5));

// EDGES
draw((x11, y11) -- (x12, y12)); // 1-2
draw((x11, y11) -- (x13, y13)); // 1-3
draw((x12, y12) -- (x13, y13)); // 2-3

// LABELS
label("$\alpha$", ((x11+x12)/2, (y11+y12)/2 + 7)); // 1-2
label("$\beta$", ((x12+x13)/2 + 7, (y12+y13)/2)); // 2-3
label("$\gamma$", ((x11+x13)/2, (y11+y13)/2 - 7)); // 1-3
label("$a$", (x11-13, y11));    // 1
label("$b$", (x12+7,  y12+10)); // 2
label("$c$", (x13+7,  y13-10)); // 3

// --- graph 2
// VERTICES
real s = 150;
real x21 = x11 + s;  real y21 = y11;  fill(circle((x21,y21),3));  // 1
real x22 = x12 + s;  real y22 = y12;  fill(circle((x22,y22),3));  // 2
real x23 = x13 + s;  real y23 = y13;  fill(circle((x23,y23),3));  // 3
draw(circle((x21,y21),5));

// EDGES
draw((x21, y21) -- (x22, y22)); // 1-2
draw((x22, y22) -- (x23, y23)); // 2-3

// LABELS
label("$\alpha$", ((x21+x22)/2, (y21+y22)/2 + 7)); // 1-2
label("$\beta$", ((x22+x23)/2 + 7, (y22+y23)/2)); // 2-3
label("$a$", (x21-13, y21));    // 1
label("$b$", (x22+7,  y22+10)); // 2
label("$c$", (x23+7,  y23-10)); // 3

label("$G$",  ((x11+x12)/2, (y11+y13)/2 - 35));
label("$G'$", ((x21+x22)/2, (y21+y23)/2 - 35));
\end{asy}
\end{center}
\caption{Counterexample for "generalized" Lemma \ref{lem:sg1_1}}
\label{fig:counterex_fc}
\end{figure}
\vvv

Here $\alpha$, $\beta$ and $\gamma$ are some arbitrary positive numbers -- values of cost function $f$. Solving system \ref{eq:maineq_graphf} we obtain $H^f_G(c,a) = \beta + (\alpha+2\gamma)/3$ and $H^f_{G'}(c,a) = \alpha+3\beta$ and, if $\gamma > \alpha+3\beta$ then hitting time for vertex $c$ has actually decreased after deleting edge $\edge{a,c}$.

\vspace{10pt}

However we can still reuse some ideas from Lemma \ref{lem:sg1_1}. Let us define linear operator $\mathbf{P_G}: \mbr^m \rightarrow \mbr^n$ by formula
$$
\mathbf{P_G}:f\rightarrow s=\{s_x\}, \quad s_x = \sum_{z \in \mcn(x)}p_{xz}f(xz)
$$

Thus on the right side of system \ref{eq:maineq_graphf} we have vector $\mathbf{P_G}(f)$ with $f$ considered as a vector in $\mbr^m$. Therefore, for every vertex $x$ we have 
$$
\mch^f_G(x) = (\mathbf{L}^{-1} \mathbf{P_G}) (f) = \sum_{e\in\mce}h_{x, e}f(e)
$$
that is, hitting time for vertex $x$ is a linear combination of edge costs with some coefficients that depend only on graph $G$, vertex $x$ and edge $e$. We will denote these coefficients as $h_{x,e}$ and the preceding equation serves as their definition. To finalize the proof we need to show that the following inequality
\begin{align}
\label{eq:hve}
h_{x,e} \lte 2\,\dist_a(e)+1
\end{align}
holds true for any vertex $x$ and edge $e$.

To start with, it is obvious that $h_{x,e}$ is the same as expected value of the number of times that random walk starting at $x$ passes through edge $e$ before it reaches $a$. To show that, simply use vector (edge-cost function) $f(e') = \delta_{ee'}$.

Thus, if $\dist_a(e) = 0$ (that is, $e$ and $a$ are incident) then $h_{x,e} \lte 1$ (you cannot walk through $e$ more than once) proving inequality \ref{eq:hve} for this case.

Now for any vertex $x\in \mcv$ or any edge $e\in \mce$ we will define functions
\begin{align}
\mbs_x(v,a)&: \mcv \times \mcv \rightarrow \mathbb{Z}^+ \nncr
\mbs_e(v,a)&: \mcv \times \mcv \rightarrow \mathbb{Z}^+ \nn
\end{align}
as the expected value of the number of times that random walk starting from $v$ and stopping having reached $a$ will visit $x$ or pass through $e$, respectively. In case when vertex $a$ is fixed we will  use simplified notation $\mbs_x(v)$ or $\mbs_e(v)$. 

So we can reformulate our theorem as inequality $\mbs_e \lte 2\dist_a(e)+1$. If $e = \edge{x,z}$ then consider function $\Omega = p_{xz}\mbs_z + p_{zx}\mbs_z - \mbs_e$. It is very easy to check that
\begin{align}
\label{eq:omega}
\Omega(v) - \displaystyle{\sum_{u \in \mcv}p_{vu}\cdot \Omega(u)} & = 0, \quad \forall v\in \mcv
\end{align}
From that it is not difficult to see that $\Omega \equiv 0$. First, $\Omega(a) = 0$. Second, if $\Omega$ is not a zero function then at some vertex $v^*\neq a$ we have $|\Omega(v^*)|$ reaching its maximum. From \ref{eq:omega} it follows then that $\Omega(u)$ must have the same value as $\Omega(v^*)$ for all $u\in \mcn(v^*)$, and then $\Omega$ has the same value in all the neighbors of those vertices as well etc. Since $G$ is connected then $\Omega$ must be constant non-zero function on $G\backslash a$ which gives us an obvious contradiction with \ref{eq:omega}.

Equation \ref{eq:omega} basically says that the value of the function in a vertex is equal to the mean of the values in its neighbors. This is a so-called \textit{harmonicity} equation (or property).

Therefore $\mbs_e = p_{xz}\mbs_z + p_{zx}\mbs_z$ (or we could simply say that any transition through edge $e$ involves either walking to $x$ and then transitioning from $x$ to $z$, or walking to $z$ and transitioning from $z$ to $x$).

\begin{mlemma}
\label{lem:sb_ineq} For any vertex $x \neq a$ function $\mbs_x$ satisfies the following inequality
$$
\mbs_x(v) \lte \deg(x)\cdot \dist_a(x)
$$
\end{mlemma}

\begin{proofL}

This inequality can indeed be proved more or less the same way we did Lemma \ref{lem:sg1_1}. First, use the same reasoning to prove that function $\mbs_x(v)$ cannot decrease when we delete any edge $\edge{a,x}$ such that $a$ and $x$ are still connected in the resulting graph. Then we choose shortest path $\Pi = \left[a_0a_1a_2\ldots a_k\right]$ from $a = a_0$ to $x = a_k$ (where $k = \dist_a(x)$), and remove all edges coming out of $a$ except $\edge{a,a_1}$. In this "updated" graph we can add up expressions $\mbs_x(u)-\mbs_x(v)$ along all directed edges $\dir{uv}$ to show that $\mbs_x(a_1) = \deg(x)$ (using same "grouping" trick as in Lemma \ref{lem:sg1_2}). Then "moving" along path $\Pi$ we  prove that at each step difference $\mbs_x(a_i)-\mbs_x(a_{i-1})$ is at most $\deg(x)$ and therefore $\mbs_x(x) \lte \deg(x)\cdot\dist_a(x)$. Since function $\mbs_x(u)$ obviously attains its maximum in $u=x$, the lemma is therefore proved.
\end{proofL}

From this immediately follows the theorem's proof. 
\begin{align}
\mbs_e(u) &= p_{xz}\mbs_x(u) + p_{zx}\mbs_z(u) \nncr
          &\lte  p_{xz}\deg(x)\cdot\dist_a(x) + p_{zx}\deg(z)\cdot\dist_a(z) \nncr
          &= \dist_a(x)+\dist_a(z) \lte 2\cdot \dist_a(e)+1 \nn
\end{align}
since $p_{xz} = 1/\deg(x)$, $p_{zx} = 1/\deg(z)$.
\end{proof}

This inequality gives us another proof of Theorem \ref{thm:sg1c}. Indeed, let us arrange all edges in $G$ by their distance from $a$ in ascending order and index them correspondingly $e_1$, $e_2$, \ldots, $e_m$. Thus for $f\equiv 1$ 
$$
\mch(G) \lte \sum_{k=1}^m\left(2\dist_a(e_k)+1\right) \lte \sum_{k = 1}^m(2k-1) = m^2.
$$
because $G$ is connected and therefore $\dist_a(e_k) \lte k-1$. Once again, equality is attained only for path graph $P_m$.

\vvv

On a separate note -- hitting probability function $\mbs_x(u)$ can be easily computed for the case of a tree. In a tree for any two vertices there is a unique non-self-intersecting path that connects them; it also serves as the shortest path between these two vertices. Let $\Pi = \left[a_0a_1a_2\ldots a_k\right]$ be such a path from $a = a_0$ to $x = a_k$ (where $k = \dist_a(x)$), and $\Phi$ be such a path from $u$ to $a$. 

\vvv
\begin{figure}[H]
\begin{center}
\begin{asy}
fill(circle((-25,-20),3));
fill(circle((-25,-80),3));
draw((-25, -20) -- (15, -40));
draw((-25, -80) -- (15, -40), dashed);
label("$x$", (-22,-93));

draw((-25, -80) -- (-45, -40));
fill(circle((-45,-40),3));
draw((-25, -80) -- (-60, -70));
fill(circle((-60,-70),3));
draw((-25, -80) -- (-55, -90));
fill(circle((-55,-90),3));

fill(circle((15,-40),3));
draw((15, -40) -- (55, -40), dashed);
fill(circle((55,-40),3));
draw((55, -40) -- (80, -40), dashed);

// elipsis
fill(circle((95,-40),1));
fill(circle((102,-40),1));
fill(circle((109,-40),1));

label("$z = a_d$", (65,-55));

draw((55, -40) -- (50, -10));
fill(circle((50,-10),3));
draw((50, -10) -- (30, 10));
fill(circle((30,10),3));
draw((50, -10) -- (70, 0));
fill(circle((70,0),3));
draw(ellipse((50, 0), 55, 25));
label("$\mathcal C$", (110, 10));
label("$u$", (40,12));

draw((125, -40) -- (165, -40), dashed);
fill(circle((165,-40),3));
draw((165, -40) -- (175, -70));
fill(circle((175,-70),3));
draw((165, -40) -- (140, -70));
fill(circle((140,-70),3));

draw((165, -40) -- (205, -40), dashed);
fill(circle((205,-40),3));
// draw(circle((205,-40),5));

label("$a$", (205,-55));

// "box" around path a--x
draw((-26,-89)--(17,-47)--(212,-47)--(212,-33)--(12,-33)--(-35,-80)--(-26,-89), dotted);
label("$path\dmhs\Pi$", (10,-66));

// "box" around path a--u
draw((50,-44)--(210,-44)--(210,-36)--(59,-36)--(55,-9)--(30,17)--(23,10)--(45,-12)--(50,-44), dotted);
label("$path\dmhs\Phi$", (22,-5));

\end{asy}
\end{center}
\end{figure}
\vvv

Then value of $\mbs_x(u)$ can be computed by the following formula
$$
\mbs_x(u) = d\cdot \deg(x), \quad d = \max\{i : a_i \in \Pi\cap\Phi\}
$$
which can be visually represented in the following manner: if we (mentally) remove path $\Pi$ from $G$ then $G\backslash\Pi$ turns into a disjoint union of several connected components such that in every one of them all vertices share the same \textit{"point of entrance into $\Pi$"}, say, $z$ ($z = a_d$). Then in all the vertices of this component function $\mbs_x$ has the same value equal to $\mbs_x(z) = d\cdot \deg(x)$.



\msection{Some results for asymmetric random walks}

Now we will turn to the case of edge-weighted functions or asymmetric random walks.

We will use a simple measure of \textit{asymmetry} (or \textit{transitional bias}) for any edge-weight function $w$ (or a random walk). We will denote it $\tau=\tau_w$; it is defined as
$$
\tau_w = \max_{x\in\mcv}\max_{y, z\in\mcn(x)}\frac{w(x,y)}{w(x,z)}
$$
or, alternatively, as
$$
\tau_w = \max\left(\frac{p_{xz}}{p_{xy}}: x,y,z\in \mcv,\: p_{xy}>0\right).
$$
That is, for each vertex we compute the maximum ratio between non-zero transition probabilities from this vertex, and then we find maximum among those values.

There is only one special case where $\tau$ is undefined for a connected graph -- when it has only one vertex. This is a trivial case and clearly we know all there is to know about this walk/graph's properties including hitting times. For all other random walks $\tau\gte 1$, and $\tau=1$ if and only if random walk is symmetric (edge-weight function is constant on $\mce$).

We will undertake a somewhat different approach although a few lemmas will be reused. Thankfully, we are not constrained by the requirements of space here, and I care much more for keeping the reader's interest alive than for brevity. I also (educator's bias, perhaps) dislike writing dense and less readable proofs of general cases instead of explaining main ideas for some natural special case and then expanding or generalizing the proof.

I should say that generally symmetric walks suffice for all standard computational algorithms. Need for asymmetric walks is rare since they seem to be of limited usability in terms of computer (or complexity theory) applications.

First of all, defining them in a suitable manner is often messy. Since random walks are  generated from \textit{edge-weighted graphs}, that means we have to be able to quickly compute values of the edge-weight function. If you define them differently that might require some considerable extra storage which has to be accessed at every turn.

Also, as we will shortly see, upper bounds for hitting time involve exponential functions of $m$ such as $\tau^{m-1}$. Graphs to which we apply the algorithms of this sort usually contain many thousands (or even millions) of vertices and edges, therefore making a $\tau^{m-1}$-type estimate almost absolutely useless. Still, such random walks present an interesting challenge at least mathematically; and perhaps the facts we will prove might turn out to be of some use for the algorithm theory in the future.

\xxx

Now let us describe in a nutshell the very useful connection between random walks on finite graphs and electric networks.

Electric network $E$ is simply a finite connected undirected graph in which each undirected edge $e = \edge{u,v}$ has positive \textit{resistance} $r_e = r_{uv}>0$ (and \textit{conductance} $c_e = c_{uv} = 1/r_{uv}$). We allow current to flow between vertices of this construct by, for example, setting voltage at vertex $a$ to zero, and voltage at $x$ to 1. Of course the current flow has to comply with basic laws of electricity such as Kirchhoff's and Ohm's laws. In "physical reality" we would select two points $a$ and $x$ and attach point $a$ to earth and point $x$ to a current source, and let the electricity flow in accordance with laws of nature.

\vvv
\begin{figure}[H]
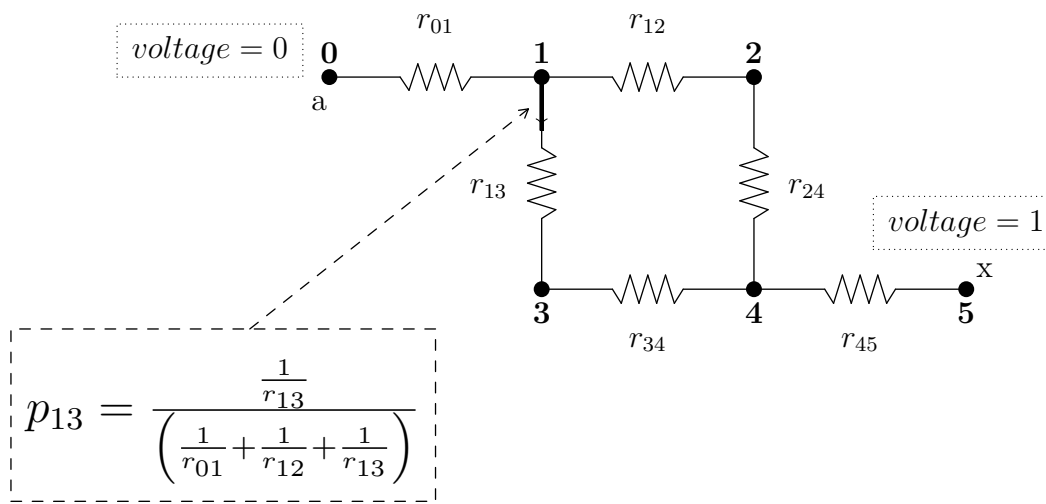

\begin{center}
\begin{asy}
// VERTICES
real x0 = 20;  real y0 = 0;   fill(circle((x0,y0),3)); // 0
real x1 = 100; real y1 = 0;   fill(circle((x1,y1),3)); // 1
real x2 = 180; real y2 = 0;   fill(circle((x2,y2),3)); // 2
real x3 = 100; real y3 = -80; fill(circle((x3,y3),3)); // 3
real x4 = 180; real y4 = -80; fill(circle((x4,y4),3)); // 4
real x5 = 260; real y5 = -80; fill(circle((x5,y5),3)); // 4

real dx3 = (x1-x0)/3;
real dy3 = (y1-y3)/3;
real dxs = dx3/8;
real dys = (y3-y1)/16;

// EDGES
// draw((x0, y0) -- (x1, y1)); // 0-1
draw((x0, y0) -- (x0+dx3, y1));
draw((x0+2*dx3, y0) -- (x1, y1));
real xc = x0+dx3;
real yc = y0;
for (real ii = 0; ii < 8; ii = ii+1)
{
	if (ii == 7)
		dys = 0;
	draw((xc, yc) -- (xc+dxs, y0+dys));
	xc = xc + dxs;
	yc = y0+dys;
	dys = -dys;
}

// draw((x1, y1) -- (x2, y2)); // 1-2
draw((x1, y1) -- (x1+dx3, y2));
draw((x1+2*dx3, y1) -- (x2, y2));
xc = x1+dx3;
yc = y1;
dys = (y3-y1)/15;
for (real ii = 0; ii < 8; ii = ii+1)
{
	if (ii == 7)
		dys = 0;
	draw((xc, yc) -- (xc+dxs, y0+dys));
	xc = xc + dxs;
	yc = y0+dys;
	dys = -dys;
}

// draw((x3, y3) -- (x4, y4)); // 3-4
real dxs = dx3/8;
draw((x3, y3) -- (x3+dx3, y3));
draw((x3+2*dx3, y4) -- (x4, y4));
xc = x3+dx3;
yc = y3;
dys = (y3-y1)/15;
for (real ii = 0; ii < 8; ii = ii+1)
{
	if (ii == 7)
		dys = 0;
	draw((xc, yc) -- (xc+dxs, y3+dys));
	xc = xc + dxs;
	yc = y3+dys;
	dys = -dys;
}

// draw((x4, y4) -- (x5, y5)); // 4-5
real dxs = dx3/8;
draw((x4, y4) -- (x4+dx3, y4));
draw((x4+2*dx3, y4) -- (x5, y5));
xc = x4+dx3;
yc = y4;
dys = (y3-y1)/15;
for (real ii = 0; ii < 8; ii = ii+1)
{
	if (ii == 7)
		dys = 0;
	draw((xc, yc) -- (xc+dxs, y4+dys));
	xc = xc + dxs;
	yc = y4+dys;
	dys = -dys;
}

//draw((x1, y1) -- (x3, y3)); // 1-3
draw((x1, y1) -- (x3, y1-dy3));
draw((x3, y3+dy3) -- (x3, y3));
xc = x1;
yc = y1-dy3;
dxs = (x1-x0)/15;
dys = dy3/8;
for (real ii = 0; ii < 8; ii = ii+1)
{
	if (ii == 7)
		dxs = 0;
	draw((xc, yc) -- (x1+dxs, yc-dys));
	xc = x1+dxs;
	yc = yc-dys;
	dxs = -dxs;
}

// draw((x2, y2) -- (x4, y4)); // 2-4
draw((x2, y2) -- (x4, y2-dy3));
draw((x2, y4+dy3) -- (x4, y4));
xc = x2;
yc = y2-dy3;
dxs = (x1-x0)/15;
dys = dy3/8;
for (real ii = 0; ii < 8; ii = ii+1)
{
	if (ii == 7)
		dxs = 0;
	draw((xc, yc) -- (x2+dxs, yc-dys));
	xc = x2+dxs;
	yc = yc-dys;
	dxs = -dxs;
}

// LABELS
label("\textbf{0}", (x0, y0+9));   // 0
label("\textbf{1}", (x1, y1+9));   // 1
label("\textbf{2}", (x2, y2+9));   // 2
label("\textbf{3}", (x3, y3-9));   // 3
label("\textbf{4}", (x4, y4-9));   // 4
label("\textbf{5}", (x5, y5-9));   // 5

label("a", (x0-4, y0-10));
label("x", (x5+7, y5+7));

label("$voltage=0$", (x0-45, y0+10));   // 0
draw(box((x0-80,y0), (x0-10, y0+20)), dotted);
label("$voltage=1$", (x5, y5+25));   // 4
draw(box((x5-35,y5+15), (x5+35, y5+35)), dotted);
label("$r_{01}$", (x0+40, y0+20));    // 0-1
label("$r_{12}$", (x1+40, y1+20));    // 1-2
label("$r_{13}$", (x1-20, y1-42));    // 1-3
label("$r_{24}$", (x2+20, y2-42));    // 2-4
label("$r_{34}$", (x3+40, y3-20));    // 3-4
label("$r_{45}$", (x4+40, y4-20));    // 4-5

// BIG FAT LABEL WITH ARROWS
label(scale(1.5)*"$p_{13} = \frac{\frac1{r_{13}}}{\left(\frac1{r_{01}}+\frac1{r_{12}}+\frac1{r_{13}}\right)}$", (x0-40, y3-50));
draw(box((x0-120,y3-80), (x1-40, y3-15)), dashed);
draw((x0-30, y3-15) -- (x1-5, y1-12), arrow=ArcArrow(SimpleHead), dashed);
draw((x1, y1) -- (x1, y1-18), arrow=Arrow(TeXHead));
real thk = 0.6;
filldraw((x1-thk, y1) -- (x1-thk, y1-20) -- (x1+thk, y1-20) -- (x1+thk, y1) -- cycle, black);

\end{asy}
\end{center}
\caption{Example of an electric network}
\end{figure}
\vvv

It turns out that if we define transition probabilities for edges of $E$ in the following way
$$
p_{uv} = \frac{c_{uv}}{C_u} , \quad \mathrm{where }\ \ C_u = \sum_{w\in\mcn(u)} c_{uw}
$$
(that is, we are using conductances of the edges as their "weights") then the resulting random walk is very "intimately" connected with properties of the underlying electric network $E$. Namely, when voltages at $a$ and $x$ are fixed as described above they uniquely determine voltages in all other points of the network, and voltage $\mbv_u$ in vertex $u$ equals probability $\mbp_u$ that random walk starting in vertex $u$ will pass through vertex $x$ before reaching $a$. Many other interesting facts follow. For instance, stationary probability distribution $\mathbf{f} = \{\mathbf{f}_u\}$ for this  random walk can be computed as
$$
\mathbf{f}_u = \frac{C_u}{\sum_{w\in E}C_w}
$$
(see \cite{DoyleSnell}, or \cite{HopKan}).

The statement about voltages $\mbv_u$ can be proved using function $\Omega(u) = \mbv_u-\mbp_u$ and checking that it complies with harmonicity equations \ref{eq:omega} for vertices different from $a$ and $x$ while vanishing in $a$ and $x$. From that it easily follows that $\Omega \equiv 0$. This is the same approach (employing discrete harmonic functions on $E$ without calling them that) which we already used before in the proof of Theorem \ref{thm:fg1}. More on the harmonic functions on graphs can be read in \cite{HopKan} or \cite{LyonsPeres}.

Some well known properties of finite electric networks can be applied to prove various facts about random walks. Among such properties is Raleigh's Monotonicity Law (increasing some edge resistances in electric network can only increase effective resistances between any two points) and some of its corollaries. One of them states that cutting an edge (with the network staying connected) can only increase all effective resistances in the network. The other one says that replacing vertex $x$ with two vertices $x'$ and $x''$ connected by edge $\edge{x',x''}$ (with some positive resistance) with some of the edges $\edge{x,u}$ being "moved" to $x'$ (turning into $\edge{x',u}$) and some -- to $x''$ ($\edge{x,u}$ becomes $\edge{x'',u}$), also can only increase the resistance values. Vice versa, shrinking any edge $\edge{x,y}$ into one vertex $z$ -- with corresponding changes in adjacency and incidence -- can only decrease remaining effective resistance values (and therefore, increase conductance). 

\xxx

We will use two simple functions of two variables $m$ and $t$, where $m$ will later represent the number of edges in the graph, and $t$ will represent walk's asymmetry $\tau$. 
\begin{align}
\mcf(t, m) \dmhs &= \dmhs \sum^{m-1}_{k=0}c_kt^k \dmhs = \dmhs \frac{2t^{m+1}-mt^2-2t+m}{(t-1)^2} \nncr
\mcp(t, m) \dmhs &= \dmhs 2\sum_{k=1}^{m-1}t^k+1 \dmhs = \dmhs \frac{2(t^m-1)}{t-1}-1     \nn
\end{align}
where $c_0 = m$, $c_k = 2(m-k)$, $k = 1,\ldots, m-1$. Both $\mcf$ and $\mcp$ are integer polynomials of $t$ with coefficients which are integer polynomials of $m$.

We will need the following easily deducible properties of functions $\mcf$ and $\mcp$. 

\begin{enumerate}[a)]
\item Both functions are positive and monotonically increasing for $t\gte0$ and $m\gte0$; \label{itm:fp_mono}
\vspace{3pt}
\item $\mcf(t, m) = \sum_{k=1}^m\mcp(t, k)$ for $m\gte1$;  \label{itm:fp_fsump}
\vspace{3pt}
\item $\mcp(t, a+b) \gte \mcp(t, a) + \mcp(t, b) + 1 $ for $t \gte 1$; \label{itm:fp_sum}
\vspace{3pt}
\item $\mcp(t, m+1) = t\mcp(t, m) + t + 1 $.  \label{itm:fp_pnext}
\end{enumerate}

Now let us begin with path tree $P_m$. We denote probability of moving left (away from $a_0$) for vertex $a_k$ as $p_k$ and probability of moving right (toward $a_0$) as $q_k$, where $q_k=1-p_k, \: k > 0$. In other words, $p_{a_ka_{k+1}} = p_k, \: m > k \gte 0$ and $p_{a_ka_{k-1}} = q_k, \: 0 < k \lte m$.

\begin{mprop}
\label{prop:pt_ineq}
For any random walk on $P_m$ the inequality for maximum hitting time $\mch(P_m) \lte \mcf(\tau, m)$ holds true.
\end{mprop}

\begin{proofP}
Without loss of generality we can assume that absorbing vertex $a$ coincides with $a_0$ -- that is, with the rightmost end of the path tree.

Then we will claim a slightly better result. Namely we will prove the inequality not for the walk's asymmetry $\tau$ but for 
$$
\tilde{\tau} = \max(1, \{\frac{p_k}{q_k}: 0 < k < m\})
$$
Obviously, $\tilde{\tau} \lte \tau$, and with function $\mcf$ being monotonic this will prove the desired result.

\vvv
\begin{figure}[H]
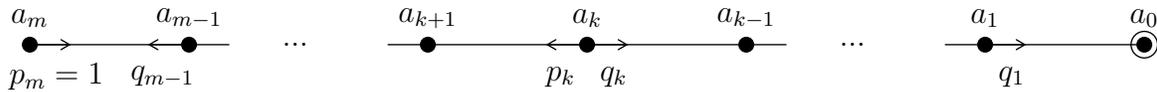

\begin{center}
\begin{asy}
// VERTICES
real x0 = 0;    real y0 = -20;  fill(circle((x0,y0),3));  // m
real x1 = 60;   real y1 = -20;  fill(circle((x1,y1),3));  // m-1
real x2 = 150;  real y2 = -20;  fill(circle((x2,y2),3));  // k+1
real x3 = 210;  real y3 = -20;  fill(circle((x3,y3),3));  // k
real x4 = 270;  real y4 = -20;  fill(circle((x4,y4),3));  // k-1
real x5 = 360;  real y5 = -20;  fill(circle((x5,y5),3));  // 1
real x6 = 420;  real y6 = -20;  fill(circle((x6,y6),3));  // 0
draw(circle((x6,y6),5));

// EDGES
draw((x0, y0) -- (x1, y1));    // m - m-1
draw((x1, y1) -- (x1+15, y1)); // m-1 ->...
draw((x2-15, y2) -- (x2, y2)); // ... -> k+1
draw((x2, y2) -- (x3, y3));    // k+1 - k
draw((x3, y3) -- (x4, y4));    // k - k-1
draw((x4, y4) -- (x4+15, y4)); // k-1 ->...
draw((x5-15, y5) -- (x5, y5)); // ... -> 1
draw((x5, y5) -- (x6, y6));    // 1 - 0

// LABELS
label("$a_m$",     (x0, y0+10)); // m
label("$a_{m-1}$", (x1, y1+10)); // m-1
label("...",       (x1+40, y1)); // m
label("$a_{k+1}$", (x2, y2+10)); // k+1
label("$a_k$",     (x3, y3+10)); // k
label("$a_{k-1}$", (x4, y4+10)); // k-1
label("...",       (x4+40, y4)); // m
label("$a_1$",     (x5, y5+10)); // 1
label("$a_0$",     (x6, y6+10)); // 0

// ARROWS
draw((x0, y0) -- (x0+15, y0),arrow=ArcArrow(SimpleHead));   // m -> m-1
label("$p_m=1$",     (x0+10, y0-12));
draw((x1, y1) -- (x1-15, y1),arrow=ArcArrow(SimpleHead));   // m-1 -> m
label("$q_{m-1}$",   (x1-10, y1-12));
draw((x3, y3) -- (x3-15, y3),arrow=ArcArrow(SimpleHead));   // k -> k+1
label("$p_k$",   (x3-10, y3-12));
draw((x3, y3) -- (x3+15, y3),arrow=ArcArrow(SimpleHead));   // k -> k-1
label("$q_k$",   (x3+10, y3-12));
draw((x5, y5) -- (x5+15, y5),arrow=ArcArrow(SimpleHead));   // 1 -> 0
label("$q_1$",   (x5+10, y5-12));

\end{asy}
\caption{Asymmetric one-dimensional finite random walk}
\end{center}
\end{figure}
\vvv

If we denote $h_k = \mch(a_k)$, and $d_k = h_{k+1}-h_k$, $t_k = p_k/q_k$ then we have
\begin{align}
h_k &= 1 + p_kh_{k+1}+q_kh_{k-1},\quad 0 < k < m \nncr
d_{k-1} &= \frac{1+p_kd_k}{q_k} = (1+t_k)+t_kd_k \lte (1+\tilde{\tau})+\tilde{\tau}d_k \nn 
\end{align}

Using this inequality recursively with $d_{m-1} = 1$, we come to our next inequality
$$
d_{m-k} \lte 2(\tilde{\tau}^{k-1}+\tilde{\tau}^{k-2}+\ldots+\tilde{\tau})+1 = \mcp(\tilde{\tau}, k),\quad 0 < k < m
$$ 
and then 
$$
\mch(P_m) = h_m = d_0+d_1+\ldots+d_{m-1} \lte \sum^m_{k=1}\mcp(\tilde{\tau}, k) = \mcf(\tilde{\tau}, m)
$$
\end{proofP}

Now we are ready to state and prove the case of an arbitrary tree.

\begin{mthm}
\label{thm:ast1}
For any finite tree $T$ with $m$ edges and any random walk on $T$ with asymmetry $\tau$ inequality $\mch(T) \lte \mcf(\tau, m)$ holds true.
\end{mthm}

\begin{proof}

Let us fix some vertex $a$. Then to prove that for any vertex $x$ we have $\mch(x) = \mch(x,a) \lte \mcf(\tau, m)$ we start with performing our usual conversion to an $\mca$-walk with absorbing vertex $a$.

Then, just as it was with Theorem \ref{thm:st1} the main step in this proof is represented by the following lemma which is a generalization of Lemma \ref{lem:tail_eq}.

\begin{mlemma}
\label{lem:ast1_1}
In the conditions of Lemma \ref{lem:tail_eq} the following inequality is true
\begin{align}
\label{eq:tail_ineq}
\mch(y) - \mch(x) &\lte \mcp(\tau, \ell+1)
\end{align}
where $\ell = \ell(y)$ is the number of edges in the tail $\mcl(y)$ of vertex $y$.
\end{mlemma}

\begin{proof}
The proof is quite similar to the one of Lemma \ref{lem:tail_eq}. Again we will do it by induction by $\ell(y)$. Basis $\ell(y) = 0$ is obvious. 

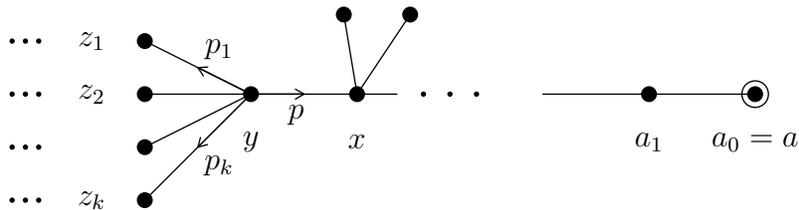
\vvv
\begin{figure}[H]
\begin{center}
\begin{asy}
fill(circle((-75,-20),1));
fill(circle((-70,-20),1));
fill(circle((-65,-20),1));
fill(circle((-75,-40),1));
fill(circle((-70,-40),1));
fill(circle((-65,-40),1));
fill(circle((-75,-60),1));
fill(circle((-70,-60),1));
fill(circle((-65,-60),1));
fill(circle((-75,-80),1));
fill(circle((-70,-80),1));
fill(circle((-65,-80),1));

fill(circle((-25,-20),3));
fill(circle((-25,-40),3));
fill(circle((-25,-60),3));
fill(circle((-25,-80),3));

draw((-25, -20) -- (15, -40));
label("$z_1$", (-45,-20));
draw((-25, -40) -- (15, -40));
label("$z_2$", (-45,-40));
draw((-25, -60) -- (15, -40));

draw((-25, -80) -- (15, -40));
label("$z_k$", (-45,-80));

draw((15, -40) -- (-5, -30), arrow=ArcArrow(SimpleHead));
label("$p_1$", (3,-23));
draw((15, -40) -- (-5, -60), arrow=ArcArrow(SimpleHead));
label("$p_k$", (3,-67));

fill(circle((15,-40),3));
draw((15, -40) -- (55, -40));
draw((15, -40) -- (35, -40), arrow=ArcArrow(SimpleHead));
label("$p$", (32,-48));

fill(circle((55,-40),3));
draw((55, -40) -- (70, -40));

fill(circle((80,-40),1));
fill(circle((90,-40),1));
fill(circle((100,-40),1));
label("$y$", (15,-58));
label("$x$", (55,-58));

draw((55, -40) -- (75, -10));
fill(circle((75,-10),3));
draw((55, -40) -- (50, -10));
fill(circle((50,-10),3));

draw((125, -40) -- (165, -40));
fill(circle((165,-40),3));
draw((165, -40) -- (205, -40));
fill(circle((205,-40),3));
draw(circle((205,-40),5));

label("$a_0 = a$", (205,-58));
label("$a_1$", (165,-58));
\end{asy}
\end{center}
\end{figure}
\vvv

Let us denote transition probabilities $p_{yz_i}$ as $p_i$, and $p_{yx}$ as $p$. Then  (\ref{eq:maineq_graph}) can be rewritten in the following manner
\begin{align}
\mch_y &= \sum^k_{i=1} p_i\mch_{z_i} + p\mch(x) + 1 \nncr
(p + \sum^k_{i=1} p_i)\mch_y &= \sum^k_{i=1} p_i\mch_{z_i} + p\mch(x) + 1 \nncr
p(\mch(y)-\mch(x)) &= \sum^k_{i=1} p_i(\mch_{z_i}-\mch(y)) + 1 \nncr
\mch(y)-\mch(x) &= \sum^k_{i=1}\frac{p_i}p(\mch_{z_i}-\mch(y)) + \frac1p. \nn
\end{align}
Thus
$$
\mch(y)-\mch(x) \lte \tau\sum^k_{i=1}(\mch_{z_i}-\mch(y)) + k\tau + 1 \nn
$$
and using induction hypothesis, properties (\ref{itm:fp_sum}-\ref{itm:fp_pnext}) of functions $\mcp$ and $\mcf$, inequality $\tau\gte1$ and the fact that sum of $\ell(z_i)$ equals $\ell-k$ we have 
\begin{align}
\mch(y)-\mch(x) &\lte \tau\sum^k_{i=1}\mcp(\tau, \ell(z_i)+1) + k\tau + 1 \nncr
                &\lte \tau\mcp(\tau, \ell(z_1)+\ldots+\ell(z_k)+k) + \tau + 1 \vphantom{\sum^k_{i=1}} \nncr
                & = \tau\mcp(\tau, \ell)+\tau+1 = \mcp(\tau, \ell+1) \vphantom{\sum^k_{i=1}}.\nn
\end{align}
\end{proof}

Now we can finalize the proof of Theorem \ref{thm:ast1}. Let us connect any vertex $x$ in $T$ with absorbing vertex $a$ by non-self-intersecting path of length $d$ indexing its vertices as $a_0 = a$, $a_1$, \ldots, $a_d = x$. Denoting $\ell(a_k)$ as simply $\ell_k$, writing out inequalities (\ref{eq:tail_ineq}) and adding them up we have
\begin{align}
\mch(a_k)-\mch(a_{k-1}) &\lte \mcp(\tau, \ell_k+1)\ ,\quad k = 1\ldots d \nncr
\mch(x) = \mch(a_d) = \sum_{k=1}^d \mcp(\tau, \ell_k+1) &\lte \sum_{k=1}^m \mcp(\tau, k) = \mcf(\tau, m) \nn
\end{align}
since numbers $\ell_k+1$ are a monotonically decreasing sequence of $d \lte m$ different positive integers with the largest of them no greater than $m$.
\end{proof}

Similarly to Proposition \ref{prop:pt_ineq} we have actually proved a slightly stronger fact. In a tree with one vertex $a$ marked (in our case, the absorbing vertex) for any edge $e=\edge{u,v}$ we can define direction "towards" $a$ on it (or alternatively, "away" from $a$). Graph $G\backslash e$ has exactly two components. Namely, direction on $e$ "towards" $a$ is direction towards that vertex out of $u$ and $v$ which lies in the same component as $a$. Any vertex in the tree (except $a$) has exactly one incident edge directed \textit{towards} $a$, all the others are pointing \textit{away} from it.

If we define $\tilde{\tau}$ as maximum of 1 and all ratios $p/q$ between two probabilities for transition from same vertex where $p$ is probability of an "away" transition, and $q$ -- of "towards" transition, then it is easy to see that we have actually proved our inequality for $\tilde{\tau}$ instead of $\tau$. As a nice corollary we obtain that if for any vertex its transition probability "towards" $a$ is greater than or equal to any transition probability "away" from $a$ then $\tilde{\tau} = 1$ and thus we will have $\mch(G) \lte m^2$.

Obviously, this corollary cannot be directly generalized for an arbitrary connected graph as notion of "direction" cannot be similarly defined in a graph with cycles.

We can see that polynomial $\mcf$ gives us exact value of $\mch(G)$ only in case of path graph with constant transition probabilities $(p,q)$ where $p$ is probability of transition that moves the walker away from the absorbing vertex (which is fixed as one of the ends of the path), and $p\gte q = 1-p$. Also it is easy to show that $m^2\tau^{m-1}$ can be used as a much simpler but less precise upper bound for $\mch$.

\xxx

Let us move on to the general case of random walks (possibly asymmetric) on finite connected graphs. We will venture a guess that results similar to Theorems \ref{thm:fg1} and \ref{thm:ast1} are true for any finite connected graph and any random walk defined on it.

\begin{mthm}
\label{thm:asg1}
For any random walk on $G$ with asymmetry $\tau$ inequality $\mch(G) \lte \mcf(\tau, m)$ holds true.
\end{mthm}

\begin{proof}

As before, we choose any vertex $a$ and convert our random walk into an $\mca$-walk. Then, the following lemma is a straightforward generalization of Lemma \ref{lem:sg1_1}.

\begin{mlemma}
\label{lem:asg1_1}
Consider absorbing vertex $a$ and edge $e=\edge{a,x} \in \mce$. We will consider random walk on graph $G' = G\backslash e$ generated from the same edge-weight function restricted to $G'$. Then for any vertex $u\in\mcv$
$$\mch_G(u) \lte \mch_{G'}(u).
$$
(In other words, removing edge $e = \edge{a,x}$ and proportionally redistributing its transition probability between other edges coming out of $x$ cannot decrease hitting time.)
\end{mlemma}

\begin{proof}
Proof is almost exactly the same as for Lemma \ref{lem:sg1_1} and I will skip it. Of course, if $u$ is not in the same component of connectedness of $G'$ as $a$ then we have $\mch_{G'}(u) = \infty$ and there is nothing to prove.
\end{proof}

And we will need something similar to Lemma \ref{lem:ast1_1}. The following lemma is the main hurdle in this proof.

\begin{mlemma}
\label{lem:asg1_2}
If absorbing state $a$ is a pendant vertex connected only with vertex $a_1$ then 
\begin{align}
\mch(a_1) &\lte \mcp(\tau, m) \nn
\end{align}
\end{mlemma}

\begin{proofL}
We revert back to considering original random walk so $a$ is no longer an absorbing vertex. There is only one transition out of $a$ and $p_{aa_1} = 1$.

Let us remind you that commute time $\mcc(u,v)$ between vertices $u$ and $v$ in graph $G$ is the sum of hitting times $\mch(u,v)+\mch(v,u)$. Since $\mch(a, a_1) = 1$ it would suffice to prove that commute time between $a$ and $a_1$ is at most $\mcp(\tau, n)+1$.

To do that we will use electric network approach described earlier in this section. The following result for commute time is proved in \cite{Ruzzo} (Theorem 2.2, case of trivial cost function): for any two vertices of $\mcv$ we have equality
\begin{align}
\label{eq:ctfres}
\mcc(u,v) = \mch(u,v)+\mch(v,u) = F\cdot Res_{uv},
\end{align}
where $Res_{uv}$ is effective resistance between vertices $u$ and $v$, and
$$
F = \sum_{e\in \dir{\mce}}\frac1{r_e}.
$$
where sum is taken over the set $\dir\mce$ of all directed edges (so each undirected edge $\edge{u,v}$ gives us two terms -- for $e=\dir{uv}$ and $e=\dir{vu}$).

From equation (\ref{eq:ctfres}) we have
$$
\mch(a_1) = \mch(a_1,a) = \mcc(a,a_1) - \mch(a,a_1) = F\cdot Res_{aa_1} - 1 = F - 1
$$
since $\mch(a,a_1) = 1$, and $Res_{aa_1} = 1$ because these two vertices are connected by one edge of resistance 1 with no other edges coming out of $a$.

Now all we need is to prove that
$$
\sum_{e\in \dir{\mce}}c_e = \sum_{e\in \dir{\mce}}\frac1{r_e} = F \lte \mcp(\tau, m)+1 = 2\left(\tau^{m-1}+\tau^{m-2}+\ldots+1\right)
$$
or, switching from directed edges to undirected 
$$
\sum_{e\in \mce}c_e = \sum_{e\in \mce}\frac1{r_e} \lte \tau^{m-1}+\tau^{m-2}+\ldots+1.
$$

Consider any two incident edges $\edge{u,v}$ (or $\edge{v,u}$) and $\edge{v,w}$. Ratio of their weights (conductances) is the same as ratio $p_{vu}/p_{vw}$ which is bounded from above by $\tau$. Thus if $d = \dist_a(uv)$ then conductance of $\edge{u,v}$ cannot be greater than $\tau^d$, which together with monotonicity of function $y=\tau^x$ proves the required inequality. This also proves that it turns into equality only for a path graph.
\end{proofL}

Once again we will make use of induction by the number of edges. Basis case $m=0$ is obvious. Now we do more or less the same as in Theorem \ref{thm:sg1}. Connect vertex $x$ with absorbing vertex $a$ by any path $\Pi$ and let $a_1$ be next to last vertex of $\Pi$. Remove all edges out of vertex $a$ except for $\edge{a,a_1}$ one by one and adjust the walk accordingly as shown in Lemma \ref{lem:asg1_1}. The lemma guarantees us that for this new graph $G'$ we have $\mch_G(x,a) \lte \mch_{G'}(x,a)$. Thus proving our inequality for $G'$ will prove it for the original graph $G$, because both number of edges and $\tau$ have not increased when we switched from $G$ to $G'$.

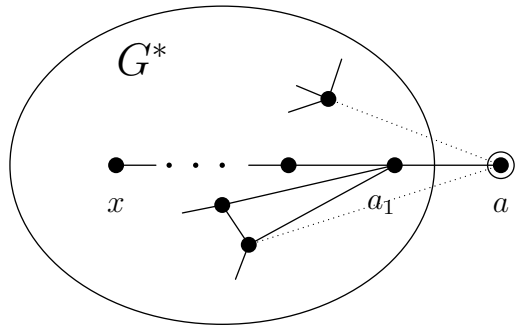
\vvv
\begin{figure}[H]
\begin{center}
\begin{asy}
// ELLIPSIS
	real xp = 80; real yp = -20;
fill(circle((xp,yp),1));
fill(circle((xp+10,-20),1));
fill(circle((xp+20,-20),1));

// VERTICES
	real x0 = xp+45; real y0 = yp;
	real x1 = xp+60; real y1 = yp+25;
	real x2 = xp+20; real y2 = yp-15;
	real x3 = xp+85; real y3 = yp;
	real x4 = xp+125; real y4 = yp;
	real x5 = xp-20; real y5 = yp;
	real x6 = xp+30; real y6 = yp-30;
fill(circle((x0, y0),3));
fill(circle((x1, y1),3));
fill(circle((x2, y2),3));
fill(circle((x3, y3),3));
fill(circle((x4, y4),3));
fill(circle((x5, y5),3));
fill(circle((x6, y6),3));
draw(circle((x4, y4),5));

// ELLIPSE
draw(ellipse((xp+20, yp), 80, 60));
label("\Large{$G^*$}", (xp-10, yp+40));

// EDGES
draw((x0-15, y0) -- (x0, y0));
// draw((x1, y1) -- (x3, y3));
draw((x2, y2) -- (x3, y3));
draw((x0, y0) -- (x3, y3));
draw((x3, y3) -- (x4, y4));
draw((x5, y5) -- (x5+15, y5));
draw((x2, y2) -- (x6, y6));
draw((x3, y3) -- (x6, y6));
draw((x4, y4) -- (x1, y1), dotted);
draw((x4, y4) -- (x6, y6), dotted);
draw((x6, y6) -- (x6-5, y6-13));
draw((x2, y2) -- (x2-15, y2-3));
draw((x1, y1) -- (x1-12, y1+5));
draw((x1, y1) -- (x1-15, y1-5));
draw((x1, y1) -- (x1+5, y1+15));

// LABELS
label("$a_1$", (x3-5, y3-15));
label("$a$", (x4, y4-15));
label("$x$", (x5, y5-15));

\end{asy}
\end{center}
\end{figure}
\vvv

Let us denote $G'\backslash a$ by $G^*$ and set $a_1$ as absorbing vertex. Then as we already know, solution of system (\ref{eq:maineq_graph}) for $G^*$ is the same as solution for $G'$ from which $\mch_{G'}(a_1,a)$ is subtracted. 

Since by induction hypothesis we have $\mch_{G^*}(x,a_1) \lte \mcf(\tau,m-1)$ then from Lemma \ref{lem:asg1_2} we have that
\begin{align}
\mch_{G'}(x,a) = \mch_{G^*}(x,a_1) &+ \mch_{G'}(a_1,a) \nncr
   &\lte \mcf(\tau,m-1) + \mcp(\tau,m) = \mcf(\tau,m). \nn
\end{align}
\end{proof}



\msection{Proof of the main theorem}

Finally, here is our strongest result that generalizes almost all of the previous ones. In a way it would have been simpler to simply state it in the very beginning, prove it and be done. However, as it was already mentioned before, I do not favor such an approach.

\begin{mthm}
\label{thm:fsgc}
For any random walk on $G=(\mcv,\mce)$ with asymmetry $\tau$ and any non-negative edge-cost function $f$ inequality
$$
\mch^f_G(x,a) \lte \sum_{e\in \mce} \mcp(\tau, \dist_a(e)+1)\cdot f(e)
$$
holds true for any vertices $a, x\in \mcv$.
\end{mthm}

\begin{proof}
Using $\mbs$-functions (see proof of Theorem \ref{thm:fg1}) we can reformulate this theorem's statement as the following inequality: $\mbs_e(x)\lte \mcp(\tau, \dist_a(e)+1)$. Now all we need is an inequality similar to Lemma \ref{lem:sb_ineq}.

\begin{mlemma}
\label{lem:asb_ineq} For any vertex $x \neq a$ function $\mbs_x$ satisfies the following inequality
\begin{align}
\label{eq:sb}
\mbs_x \lte \frac1q\cdot(1+\tau+\ldots+\tau^{d-1})
\end{align}
where $d = \dist_a(x)$ and $q = p_{xy}$ is transition probability along any edge $\edge{x,y}$ such that $\dist_a(y) < d$.

\end{mlemma}

\begin{proofL}

Here is the plan: there exists path $\Pi = [xy\ldots a_1a]$ of length $d$ connecting vertices $x$ and $a$ which begins with edge $\edge{x,y}$. Following exactly same reasoning as in Lemma \ref{lem:sg1_1}, we can claim that removal of all edges coming out of $a$ except for $\edge{a,a_1}$ does not decrease values of function $\mbs_x$, and does not increase the value of $\tau$. Therefore, we can assume that $a$ is a pendant vertex and thus $\mbs_x(v,a) = \mbs_x(a_1,a) + \mbs^*_x(v,a_1)$ where $\mbs^*_x = \mbs_{x,G^*}$ is $\mbs$-function for graph $G^* = G\backslash a$. Thus we can replace $G$ with $G^*$, use inequality \ref{eq:sb} for $G^*$ (induction by $d$) and try to get an upper bound for $\mbs_x(a_1) = \mbs_x(a_1,a)$. Let's go ahead and execute this plan.

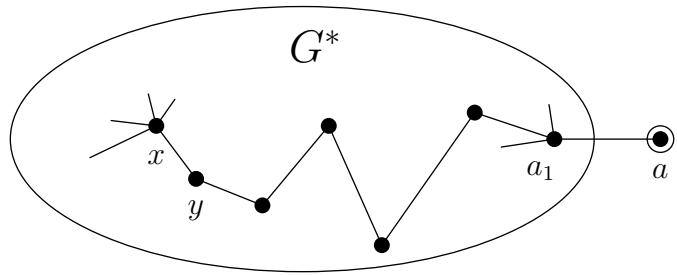
\vvv
\begin{figure}[H]
\begin{center}
\begin{asy}
// MAIN VARIABLES
real xp = 0; real yp = 0;

// VERTICES
real bx = xp-65; real by = yp+5;
real cx = xp-50; real cy = yp-15;
real dx = xp-25; real dy = yp-25;
real ex = xp;    real ey = yp+5;
real fx = xp+20; real fy = yp-40;
	
real a2x = xp+55; real a2y = yp+10;
real a1x = xp+85; real a1y = yp;
real ax = xp+125; real ay = yp;

fill(circle((bx, by),3));
fill(circle((cx, cy),3));
fill(circle((dx, dy),3));
fill(circle((ex, ey),3));
fill(circle((fx, fy),3));

fill(circle((a2x, a2y),3));
fill(circle((a1x, a1y),3));
fill(circle((ax, ay),3));
draw(circle((ax, ay),5));

// ELLIPSE
draw(ellipse((xp-10, yp), 110, 50));
label("{\Large{$G^*$}}", (xp-5, yp+35));

// "RANDOM" EDGES
draw((bx, by) -- (bx+7, by+10));
draw((bx, by) -- (bx-3, by+12));
draw((bx, by) -- (bx-17, by+2));
draw((bx, by) -- (bx-25, by-12));

draw((a1x, a1y) -- (a1x-2, a1y+13));
draw((a1x, a1y) -- (a1x-20, a1y-3));

// PATH EDGES
draw((bx, by) -- (cx, cy));
draw((cx, cy) -- (dx, dy));
draw((dx, dy) -- (ex, ey));
draw((ex, ey) -- (fx, fy));
draw((fx, fy) -- (a2x, a2y));

draw((a2x, a2y) -- (a1x, a1y));
draw((a1x, a1y) -- (ax, ay));

// LABELS
real lby = 12;
label("$x$", (bx, by-lby));
label("$y$", (cx, cy-lby));

label("$a_1$", (a1x-5, a1y-lby));
label("$a$", (ax, ay-lby));

\end{asy}
\end{center}
\end{figure}
\vvv

First, we will use induction by $d$. If $d = 1$ then $x$ is adjacent to $a$ (thus $y=a_1=a$). Therefore after removing all other edges coming out of $a$ we have 
$$
\mbs_{xa} = p_{xa}\mbs_x + p_{ax}\mbs_a = p_{xa}\mbs_x
$$
Since obviously $\mbs_{xa} = 1$, we have $\mbs_x = 1/p_{xa}$ proving the basis of induction.

Second, to prove induction step from $d-1$ to $d>1$ we need to find some upper bound for $\mbs_x(a_1)$. Let us define $W(x)$ as sum of weights for all edges coming out of vertex $x$. Second, if for every directed edge $\dir{uv}$ of graph $G$ we write product $w(u,v)\cdot(\mbs_x(u)-\mbs_x(v))$ next to that edge then sum of all these numbers is zero. But if we group them by the start vertex then for every vertex $u$ we will have the sum of the numbers in that group
\begin{align}
\sum_{v\in \mcn(u)} w(u,v)&(\mbs_x(u)-\mbs_x(v)) \nncr
  &= \mbs_x(u) W(x) - \sum_{v\in \mcn(u)} w(u,v)\mbs_x(v) \nncr
  & = W(u)\left(\mbs_x(u) -  \sum_{v\in \mcn(u)} p_{uv}\mbs_x(v)\right) \nn
\end{align}
which is zero for every vertex other than $a$ and $x$. Adding up all these grouped expressions we obtain $$
W(x) - w(a,a_1)\mbs_x(a_1) = 0
$$
and
$$
\mbs_x(a_1) = \frac{W(x)}{w(a,a_1)}
$$

Using $q = p_{xy} = \frac{w(x,y)}{W(x)}$ we get
$$
\mbs_{a_1} = \frac{W(x)}{w(a,a_1)} = \frac{w(x,y)}{q\cdot w(a,a_1)} \lte \frac{\tau^{d-1}}q
$$
because distance between edges $\edge{x,y}$ and $\edge{a,a_1}$ is $d-2$ and thus their weights' ratio is at most $\tau^{d-1}$.

Finally for any $v\in\mcv$ we have 
\begin{align}
\mbs_x(v,a) &= \mbs_x(a_1,a) + \mbs^*_x(v,a_1) \vphfpp \nncr
             &\lte \frac{\tau^{d-1}}q + \frac{1+\tau+\ldots+\tau^{d-2}}q
             = \frac{1+\tau+\ldots+\tau^{d-1}}q \nn
\end{align}
proving the induction step as well.
\end{proofL}

Now let's assume that edge $e$ connects vertices $y$ and $z$, and $d = \dist_a(e)$. We have then only two possible cases (there are actually three cases but two of them are symmetric and without loss of generality we can discard one of them).

\textsc{Case 1.} $\dist_a(y) > d$. Thus $\dist_a(z) = d$ and $\dist_a(y) = d+1$. Let $t$ be any neighbor of $z$ such that it is closer than $z$ to vertex $a$.

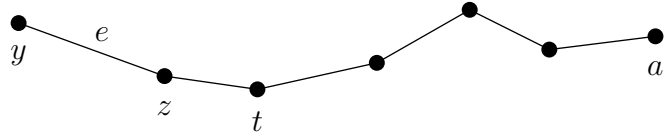
\vvv
\begin{figure}[H]
\begin{center}
\begin{asy}
// MAIN VARIABLES
real[] px = {-115,-60,-25,20,55,85,125};
real[] py = {5,-15,-20,-10,10,-5,0};
// VERTICES
int[]  idxs  = {0,1,2,3,4,5,6};
for (var ii:idxs) fill(circle((px[ii], py[ii]),3));
// PATH EDGES
int[]  idxs1 = {0,1,2,3,4,5};
for (var jj:idxs1) draw((px[jj],py[jj])--(px[jj+1],py[jj+1]));
// LABELS
real dy = 12, dx = 4;
label("$y$", (px[0], py[0]-dy));
label("$z$", (px[1], py[1]-dy));
label("$t$", (px[2], py[2]-dy));
label("$a$", (px[6], py[6]-dy));
label("$e$", ((px[0]+px[1])/2+dx, (py[0]+py[1])/2+dy/2));
\end{asy}
\end{center}
\end{figure}
\vvv

From Lemma \ref{lem:asb_ineq} we have
\begin{align}
\mbs_y &\lte \frac1{p_{yz}}\cdot(1+\tau+\ldots+\tau^d) \nncr
\mbs_z &\lte \frac1{p_{zt}}\cdot(1+\tau+\ldots+\tau^{d-1}) \nn
\end{align}
and so
\begin{align}
\mbs_e &= p_{yz}\mbs_y + p_{zy}\mbs_z \vphfpp \nncr
       &\lte (1+\tau+\ldots+\tau^d) + \frac{p_{zy}}{p_{zt}}(1+\tau+\ldots+\tau^{d-1}) \nncr
       &\lte (1+\tau+\ldots+\tau^d) + \tau(1+\tau+\ldots+\tau^{d-1}) \vphfpp \nncr
       &\lte \mcp(\tau, d+1) \vphfpp \nn
\end{align}

\textsc{Case 2.} $\dist_a(y) = \dist_a(z) = d$. Let's choose any vertex $u$ adjacent to $y$ and vertex $v$ adjacent to $z$ such that $\dist_a(u) = d-1$ and $\dist_a(v) = d-1$.

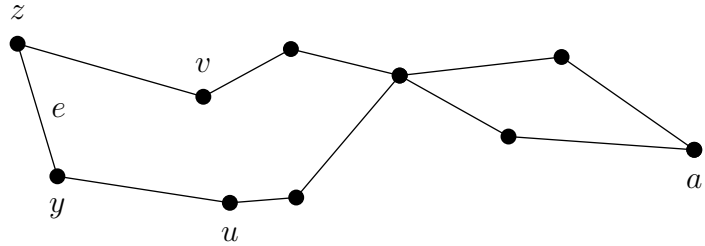
\vvv
\begin{figure}[H]
\begin{center}
\begin{asy}
// MAIN VARIABLES
real[] px = {-115,-50,-25,14,55,125};
real[] py = {-5,-15,-13,33,10,5};
real[] qx = {-130,-60,-27,14,75,125};
real[] qy = {45,25,43,33,40,5};
int[]  idxs = {0,1,2,3,4,5};
int[]  idxs1 = {0,1,2,3,4};
// A -- X path
// VERTICES
for (var ii:idxs) fill(circle((px[ii], py[ii]),3));
// PATH EDGES
for (var ii:idxs1) draw((px[ii],py[ii])--(px[ii+1],py[ii+1]));
// A -- Y path
// VERTICES
for (var ii:idxs) fill(circle((qx[ii], qy[ii]),3));
// PATH EDGES
for (var ii:idxs1) draw((qx[ii],qy[ii])--(qx[ii+1],qy[ii+1]));
// ONE MORE EDGE (XY)
draw((px[0], py[0]) -- (qx[0], qy[0]));
// LABELS
real dy = 12, dx = 8;
label("$y$", (px[0], py[0]-dy));
label("$u$", (px[1], py[1]-dy));
label("$z$", (qx[0], qy[0]+dy));
label("$v$", (qx[1], qy[1]+dy));
label("$a$", (px[5], py[5]-dy));
label("$e$", ((px[0]+qx[0])/2+dx, (py[0]+qy[0])/2));
\end{asy}
\end{center}
\end{figure}
\vvv

Again we have
\begin{align}
\mbs_y &\lte \frac1{p_{yu}}\cdot(1+\tau+\ldots+\tau^{d-1}) \nncr
\mbs_z &\lte \frac1{p_{zv}}\cdot(1+\tau+\ldots+\tau^{d-1}) \nn
\end{align}
and
\begin{align}
\mbs_e &= p_{yz}\mbs_y + p_{zy}\mbs_z \vphfpp \nncr
       &\lte \frac{p_{yz}}{p_{yu}}(1+\tau+\ldots+\tau^{d-1}) + \frac{p_{zy}}{p_{zv}}(1+\tau+\ldots+\tau^{d-1}) \nncr
       &\lte 2\tau(1+\tau+\ldots+\tau^{d-1}) < \mcp(\tau, d+1) \vphfpp \nn
\end{align}
which concludes the proof.
\end{proof}


\msection{Acknowledgments and motivation}

A simpler case of this problem (see Theorem \ref{thm:sg1c}) was posed to me as a conjecture sometime around February 20, 2016 by my old friend and colleague Alexey Kirichenko, and author wants to thank him for the opportunity to engage in fruitful discussions about some issues in computational complexity, graph theory and Markov chains, as well as for his enduring friendship and readiness to help out whenever I needed an advice.

Now for scientific motivation. This issue comes from an old and very important complexity theory question about an algorithm with limited memory to determine whether two vertices $A$ and $B$ in any given finite undirected graph $G$ can be connected with a path. Theory that investigates this is called \textit{s-t connectivity} where "s" and "t" come from conventional names ("Source" and "Target") for two vertices of the given graph which in this article we usually called $x$ and $a$. Complexity of s-t connectivity for directed graphs is called STCON, and for undirected graphs -- USTCON.

It is known that STCON is \textbf{NL}-complete, that is, a non-deterministic Turing machine with log-space memory can provide the next step for the algorithm which will eventually build the desired path (if the graph is connected). USTCON was shown to be \textbf{L}-complete (see \cite{Rein}), meaning that it can be solved by a deterministic Turing machine using logarithmic amount of memory.

The connectivity problem becomes much easier if we decide to make do with an \textit{heuristic}; in other words if we attempt to come up with an algorithm that determines "probability" of $a$ and $x$ being connected in $G$ within some preset tolerance $\varepsilon$. For instance, we could be satisfied if after the algorithm is run we can claim we know whether $a$ and $x$ are connected or not with probability greater than 0.999 ($\varepsilon = 0.001$). Now imagine that we know that the expected length of a simple random walk on $G$ with absorbing state vertex $a$ is less than some specific number $N$. Simulating a random walk on graph $G$ requires only finite memory (basically, we only need to store current location of the walker, and some trivial fixed overhead such as the number of steps and ids of our two vertices) and if starting from $x$ we haven't reached vertex $a$ after $2N$ moves, we can stop the simulation and "claim" that probability that $a$ and $x$ are connected is below $1/2$ (by Markov's inequality, see \cite{Feller}). We then repeat this $2N$-steps-long walk $k = \lceil\log_2(1/\varepsilon)\rceil$ times. If $a$ was never reached, then probability drops below $2^{-k} < \varepsilon$ and we can consequently state (with the required level of confidence) that $a$ and $x$ are not connected.

Usually a researcher proves a "big $O$"-type of asymptotic upper bound and stops there, since from the point of view of computational complexity theory the job is done -- many statements of this type can be found in classical work by Aleliunas et al, \cite{Aleliunas}. This approach is fine for pure theoretical purposes, but it is not applicable for the situation that I have just described above because we need to simulate random walk with a specific number of steps. Therefore having an upper bound of, say, $O(n^{7/4})$ or $O(\sqrt{m})$ is not very useful for real-life computer-based implementation.

Also most of the existing estimates and results on maximum (and average) hitting time are based on the number of graph's vertices $n$ (see a survey of many such results in \cite{Lovasz}). In this article we have proved some upper bounds for maximum hitting time as functions of graph's number of edges $m$ and showed that most of these upper bounds are sharp. If graph is "sparse", which in our case means that $m = o(n^{3/2})$, then this type of upper bound will likely be better than the upper bounds based on $n$, such as a well-known theorem from \cite{BrWin} stating that the maximum hitting time is less than or equal to approximately $4n^3/27$. Among such graphs are sub-graphs of $k$-dimensional grid where $k$ is some "small" number, or generally any graphs with vertex degrees bounded from above by some fixed number which is sufficiently small compared to $n$.



\vspace{10pt}

\end{document}